\documentclass[11pt]{amsart}

\usepackage{amscd,amsmath,amssymb,fancyhdr,color}
\usepackage[utf8x]{inputenc}
\usepackage{amsfonts}
\usepackage{amsthm}
\usepackage{accents}
\usepackage{graphicx}
\usepackage{float}
\usepackage{subcaption}
\usepackage{verbatim}
\usepackage{indentfirst}
\usepackage{dsfont}
\usepackage{tikz}
\usepackage{enumerate}

\usepackage{faktor}

\usepackage[backref=page]{hyperref}
\renewcommand*{\backref}[1]{}
\renewcommand*{\backrefalt}[4]{%
    \ifcase #1 (Not cited.)%
    \or        (Cited on page~#2.)%
    \else      (Cited on pages~#2.)%
    \fi}

\hypersetup{
     colorlinks   = true,
     citecolor    = magenta
}

\newcommand{\Ker}{\text{Ker}}

\newcommand{\K}{K\"ahler}

\newcommand{\Ad}{\text{Ad}}

\numberwithin{equation}{section}

\def\eqref#1{(\ref{#1})}

\newcommand{\Z}{{\mathbb Z}}
\newcommand{\C}{{\mathbb C}}
\newcommand{\R}{{\mathbb R}}

\def\1{\sqrt{-1}\:}
\newcommand{\restrict}[1]{{\big|_{{\phantom{|}\!\!}_{#1}}}}
\newcommand{\cntrct}                % contraction with a vector field
{\hspace{2pt}\raisebox{1pt}{\text{$\lrcorner$}}\hspace{2pt}}

\renewcommand{\dim}{\operatorname{dim}}

\renewcommand{\Re}{\operatorname{Re}}
\renewcommand{\Im}{\operatorname{Im}}

\newcommand{\iso}{\ensuremath{\simeq}}

\newcommand{\ie}{{\em i.e. }}
\newcommand{\eg}{{\em e.g. }}

\renewcommand{\to}{\longrightarrow}

\newcounter{Mycounter}[section]
\newcounter{lemma}[section]
\newcounter{claim}[section]

\newcounter{sublemma}[section]

\newcounter{corollary}[section]

\newcounter{theorem}[section]

\newcounter{conjecture}[section]

\newcounter{proposition}[section]

\newcounter{definition}[section]

\newcounter{example}[section]

\newcounter{remark}[section]

\newcounter{problem}[section]

\newcounter{question}[section]

\makeatletter

\@addtoreset{equation}{section}

\@addtoreset{footnote}{section}

\makeatother

\def\blacksquare{\hbox{\vrule width 5pt height 5pt depth 0pt}}

\def\endproof{\hfill\blacksquare}

\usetikzlibrary{arrows,chains,matrix,positioning,scopes}

\makeatletter
\tikzset{join/.code=\tikzset{after node path={%
			\ifx\tikzchainprevious\pgfutil@empty\else(\tikzchainprevious)%
			edge[every join]#1(\tikzchaincurrent)\fi}}}
\makeatother

\tikzset{>=stealth',every on chain/.append style={join},
	every join/.style={->}}

\begin{document}

\newpage

\title{Locally conformally symplectic reduction}
\author{Miron Stanciu}
\address{Institute of Mathematics ``Simion Stoilow'' of the Romanian Academy\newline 
	21, Calea Grivitei Street, 010702, Bucharest, Romania\newline 
\indent	{\em and}\newline
\indent	University of Bucharest, Faculty of Mathematics and Computer Science, 14 Academiei Str., Bucharest, Romania}
	\thanks{Partially supported by a grant of Ministry of Research and Innovation, CNCS - UEFISCDI,
project number PN-III-P4-ID-PCE-2016-0065, within PNCDI
III.\\[.1in]
}
\email{miron.stanciu@imar.ro; mirostnc@gmail.com}

\date{\today}

\begin{abstract}
	We present a reduction procedure for locally conformally symplectic (LCS) manifolds with an action of a Lie group preserving the conformal structure, with respect to any regular value of the momentum mapping. Under certain conditions, this reduction is compatible with the existence of a locally conformally \K \ structure. As a special consequence, we obtain a compatible contact reduction with respect to any regular value of the contact momentum mapping. 
	
\end{abstract}

\maketitle

\hypersetup{linkcolor=blue}
\tableofcontents

\section{Introduction}

The Marsden-Weinstein reduction for symplectic manifolds is a classic and well known procedure (see \cite[pp. 298-299]{am}, \cite[ch. 7]{br}) by which, given a Poisson action of a Lie group  on a symplectic manifold, one can produce other symplectic manifolds as quotients of level sets of the momentum mapping. Since its introduction, the general idea has been applied to many geometric structures (contact, K\"ahler etc.). 

In this paper we adapt the symplectic reduction to locally conformally symplectic (briefly LCS) manifolds in a natural way. LCS manifolds are quotients of symplectic manifolds by a discrete group of homotheties. Products of contact manifolds with $S^1$ are LCS (\cite{va}), but there are many other examples. Equivalently, they are even-dimensional manifolds endowed with a non-degenerate two-form $\omega$ and a closed one-form $\theta$ (called Lee form) such that $d\omega=\theta\wedge\omega$. 

There is a growing interest in LCS manifolds, concerning both their geometry and topology (\eg \cite{baz}, \cite{mur} and the references therein). A strong motivation for studying LCS manifolds is a recent result by Eliashberg and Murphy \cite{el} proving that all compact almost complex manifolds with non-trivial integer one-cohomology do have LCS structures, and hence these manifolds are much more numerous than symplectic ones.

The symplectic reduction consists in the following steps: 
\begin{itemize}
\item Using the Poisson character of the action, one constructs a momentum mapping from the symplectic manifold to the dual of the Lie algebra of the group. This will be equivariant with respect to the group action and to the coadjoint action.
\item Factoring the level set of a regular value of the momentum mapping through the action of the stabilizer of this value. The restriction of symplectic form to the level set will project to a symplectic form on the quotient.
\end{itemize}

S. Haller and T. Rybicki (\cite{hr}) adapted this symplectic reduction to LCS manifolds. The path chosen was to retain the idea of factoring by the group action itself; however, in the general case, this means the level sets of the momentum mapping no longer satisfy the hypotheses they impose. We do the opposite: we  factor the level sets of the momentum mapping, as in the symplectic case, but not through the group action, but along a foliation derived by ``twisting'' the group action. This amounts  to the known symplectic reduction if the Lee form $\theta$ is zero.

We do not ask that the group action preserve the LCS form, only the LCS structure; however, if only the latter is true, certain additional conditions have to be imposed. Specifically, we prove:

\hfill

{\bf Theorem A} (see \ref{th:LCSred}) Let $(M, \omega, \theta)$ be a connected LCS manifold and $G$ a connected Lie group acting twisted Hamiltonian on it.

Let $\mu$ be the momentum mapping and $\xi \in \mathfrak{g}^*$ a regular value. Denote by $\mathcal{F} = T\mu^{-1}(\xi) \cap (T\mu^{-1}(\xi))^\omega$ - this is a foliation along $\mu^{-1}(\xi)$. Assume that one of the following conditions is met:
\begin{itemize}
	\item The action of $G$ preserves the LCS form $\omega$.
	\item $G$ is compact, the bilinear forms
	\[
	\mathfrak{g} \times \mathfrak{g} \ni (a,b) \mapsto \xi(a) \wedge \theta_x(X_b)
	\] are zero for all $x \in \mu^{-1}(\xi)$ (here $X_b$ is the fundamental vector field corresponding to $b \in \mathfrak{g}$) and $\theta$ is exact on $\mathcal{F}$.
\end{itemize}

If $N_\xi := \faktor{\mu^{-1}(\xi)}{\mathcal{F}}$ is a smooth manifold and $\pi : \mu^{-1}(\xi) \to N_\xi$ is a submersion, then $N_\xi$ has an LCS structure such that the LCS form $\omega_\xi$ satisfies
\[
\pi^* \omega_\xi = e^f \omega_{|\mu^{-1}(\xi)}
\]
for some $f \in C^\infty(\mu^{-1}(\xi))$. 

Moreover, one can take $f = 0$ if the action preserves the LCS form. \endproof

\hfill

We then prove that this reduction along a foliation can be expressed as a quotient of a group action of the universal algebra of the stabilizer of the regular value $\xi$.

The reduction described above can produce a great number of LCS manifolds, by varying either the regular value or the LCS form in its conformal class. We show that in both cases and under certain conditions, the quotients produced are cobordant.

We also prove that, with an additional hypothesis, our construction is compatible with the existence of a complex structure \ie if the manifold $M$ is locally conformally \K \ (LCK for short), the resulting reduced manifold is also LCK. The most significant of these conditions is met if the manifold belongs to a large subclass of LCK manifolds called Vaisman.

\hfill

{\bf Theorem A} can be used in the contact context. Contact reduction at regular value zero has been well established, independently, by various authors (see \cite{al}, \cite{ge}, \cite{lo}), for some time. Lerman and Willett \cite{lw} studied the topological structure of contact quotients. Albert \cite{al} in 1989 and Willett \cite{wi} in 2002 developed different reduction procedures for non-zero regular values; the first depends on the choice of contact form for the contact structure, the latter imposes more conditions than in the symplectic analogue. They were later unified by Zambon and Zhu \cite{zz} using groupoid actions. Exploiting the relationship between contact and LCS manifolds, we derive (and get a new proof of) a contact reduction method that works for any regular value of the momentum mapping, which turns out to have the same result as the one defined by Albert. Since we present this reduction as being naturally linked to LCS reduction, this seems to be the most natural of the existing methods for non-zero contact reduction. As a byproduct, this provides a wide  class of examples for the reduction method. Specifically, the contact reduction method is described in:

\hfill

{\bf Theorem B} (see \cite{al} \& \ref{thm:redcont}) Let $(C, \alpha)$ be a connected contact manifold and $G$ a connected Lie group acting on $C$ and preserving the contact form. Denote by $R$ the Reeb field of $C$. 

Let $\mu_C$ be the momentum mapping and $\xi \in \mathfrak{g}^*$ a regular value. Let $\mathcal{F}_C$ be the foliation
\[
(\mathcal{F}_C)_x = \{ v \in T_x \mu^{-1}(\xi) \ | \ v = (X_a)_x - \xi(a) R_x \text{ for some } a \in \mathfrak{g} \},
\]
where $X_a$ is the fundamental vector field corresponding to $a \in \mathfrak{g}$.

If $C_\xi := \faktor{\mu_C^{-1}(\xi)}{\mathcal{F}_C}$ is a smooth manifold and $\pi : \mu_C^{-1}(\xi) \to C_\xi$ is a submersion, then $C_\xi$ has a natural contact structure such that the contact form $\alpha_\xi$ satisfies
\[
\pi^* \alpha_\xi = \alpha.
\]
Moreover, $(S^1 \times C_\xi, d_\theta \alpha_\xi, \theta)$ is the LCS reduction of $S^1 \times C$ with respect to the regular value $-\xi$. \endproof

%Let $C$ be a connected contact manifold with a Lie group action of $G$ that preserves the contact form. Let $\xi$ be a regular value for the momentum mapping $\mu_C$. Let $\mathcal{F}_C$ be a specific integrable foliation described in Section \ref{sectionContRed}.

%If $C_\xi := \faktor{\mu_C^{-1}(\xi)}{\mathcal{F}_C}$ is a smooth manifold and $\pi : \mu_C^{-1}(\xi) \to C_\xi$ is a submersion, then $C_\xi$ has a natural contact structure such that the pullback of the contact form on it is equal to the contact form on $C$.

%Moreover, $S^1 \times C_\xi$ is the LCS reduction of $S^1 \times C$ with respect to the regular value $-\xi$.  \endproof

\hfill

We also prove that, if the original contact manifold is Sasaki, the reduced space obtained \textit{via} this method is also Sasaki. 

\hfill

The paper is organized as follows. In Section \ref{sectionPrelim} we present basic facts about LCS and contact manifolds, needed further. 
Section \ref{sectionLcsRed} is devoted to proving Result I and determining the cobordism class of the quotients, while Section \ref{sectionLCK} describes the conditions under which our reduction is compatible with a complex structure. In Section \ref{sectionContRed} we discuss the contact and Sasaki reductions and their compatibility with the LCS and LCK reductions. Section \ref{sectionExamples} is reserved for concrete classes of examples on which we apply the techniques developed.

\section{Preliminaries}
\label{sectionPrelim}

\begin{definition}
	A manifold $M$ with a non-degenerate two-form $\omega$ is called \textit{locally conformally symplectic} (for short, \textit{LCS}) if there exists a closed one-form $\theta$ such that 
	$$d\omega = \theta \wedge \omega.$$
	 In that case, $\omega$ is called the LCS form and $\theta$ is called {\em the Lee form}\footnote{The name was given after H.C.  Lee who discussed the problem of finding integral factors for a nondegenerate two-form to be closed, \cite{lee}.} of $\omega$. 
\end{definition}

\bigskip

This notion was introduced by P. Libermann in \cite{lib} and later studied by J. Lefebvre \cite{lef} and especially I. Vaisman \cite{va}. One can see that the name ``LCS'' is justified, because the definition above is equivalent to the existence of a cover $(U_\alpha)_\alpha$ and a family of smooth functions $f_\alpha$ on each $U_\alpha$ such that $d (e^{-f_\alpha} \omega  )= 0$ (see \cite{va}).

\bigskip

It was recently proven by Y. Eliashberg and E. Murphy that LCS manifolds are very widespread:

\begin{theorem}\textnormal {(\cite{el})}
	\label{thm:eliashberg}
	Let $(M, J)$ be a closed $2n$-dimensional almost complex manifold and $[\eta]$ a non-zero cohomology class in $H_1(M, \Z)$. Then there exists a locally conformally symplectic structure $(\omega,\theta)$ in a formal homotopy class determined by $J$, with $\theta = c \cdot \eta$ for some real $c \neq 0$. 
\end{theorem}

\begin{definition}
	Let $(M, \omega, \theta)$ be an LCS manifold. Define the \textit{twisted de Rham operator}
	\[
	d_\theta = d - \theta \wedge \ : \Omega^\bullet(M) \to \Omega^{\bullet + 1}(M).
	\]
In particular, $d_\theta\omega=0$.	Since $d\theta = 0$, we have $d_\theta^2 = 0$, so this operator defines the so-called {\em twisted} or {\em Morse-Novikov} cohomology
	\[
	H^\bullet_\theta (M)= \frac{\Ker d_\theta}{\Im d_\theta}.
	\]
\end{definition}

The following fact underlines a first key difference between \textit{de Rham} and twisted cohomology:

\begin{proposition}
\label{prop:H0MN}
Assume $M$ is connected and let $\theta$ be a closed $1$-form.
\begin{enumerate}[1)]
	\item If $[ \theta ] = 0$, then $H^\bullet_\theta(M) \simeq H^\bullet(M)$.
	\item If $[ \theta ] \neq 0$, then $H^0_\theta(M) = 0$.
\end{enumerate}

\begin{proof}
	\begin{enumerate}[1)]
		\item If $\theta = df$, we define the isomorphism
		\[
		\varphi: H^\bullet_\theta(M) \to  H^\bullet(M), \varphi([\alpha]) = [e^{-f} \alpha].
		\]
		\item Assume $\theta$ is not exact and let $f \in C^\infty(M), \ d_\theta f = 0$ \ie $df = f \theta$. Let $A = \{ x \in M \ | \ f(x) = 0 \}$; we first show $A$ is an open set.
		
		Pick $x \in A$ and $V \ni x$ open such that $\theta = dg$ on $V$. Then, from $1)$, $d(e^{-g} f) = 0$, so it is constant on $V$; since $f(x) = 0$, we have $e^{-g}f = 0$ on $V$. But since $e^{-g} \neq 0$, this means $V \subset A$.
		
		If $A = \emptyset$, then we can assume $f$ is positive and obtain $\theta = d ln(f)$, a contradiction. Since $M$ is connected, this leaves $A = M$ \ie $f = 0$.
	\end{enumerate}
\end{proof}

\end{proposition}

\begin{definition}
	In a symplectic vector space $V$, given a subspace $W \le V$, we denote by
	\[
	W^\omega = \{ v \in V \ | \ \omega (v, W) = 0 \} \le V
	\]
	the $\omega$-dual of $W$. Since $\omega$ is non-degenerate, we have 
	\[
	\dim W + \dim W^\omega = \dim V.
	\]
	Note that the above sum is not necessarily direct.
\end{definition}

\begin{remark}
	The LCS condition is conformally invariant in the following sense. If $\omega$ is an LCS form on $M$ with Lee form $\theta$, then $e^f \omega$ is also an LCS form with Lee form $\theta + df$, for any $f \in \mathcal{C}^\infty(M)$. Therefore, it is sometimes more interesting to work with an LCS structure 
	\[
	[(\omega, \theta)] = \{ (e^f \omega, \theta + df) \ | \ f \in \mathcal{C}^\infty(M) \}
	\]
	on $M$ rather than with particular forms in that structure. Note that a symplectic form exists in a given LCS structure if and only if $[\theta] = 0 \in H^1(M)$; in this case, the structure is called {\em globally conformally symplectic (GCS)}. In general, we shall not consider the GCS case.
	\end{remark}

\bigskip

Another way to look at LCS manifolds is \textit{via} a convenient covering: let $\widetilde{M}$ be the universal covering of the LCS manifold $(M, \omega, \theta)$. Choose $\lambda \in \mathcal{C}^\infty(\widetilde{M})$ such that $d\lambda = \pi^* \theta$. Then we have 
\[
d(e^{-\lambda} \pi^* \omega) = 0,
\]
so $\omega_0 = e^{-\lambda} \pi^* \omega$ is a symplectic form on the universal covering. Moreover, for a $\gamma \in \Gamma = \pi_1(M)$, since $\pi^* \theta$ is $\Gamma$-invariant, we have $\lambda \circ \gamma = \lambda + r_\gamma$ for a $r_\gamma \in \R$, so we can define a representation of $\Gamma$
\[
r: \Gamma \to \R, \ r(\gamma) = r_\gamma.
\]
As $\Gamma$ can also contain $\omega_0$-symplectomorphisms, $r$ is not injective. Factoring through its kernel amounts to considering  the {\em minimal (symplectic) covering} as
\[
\overline{M} = \faktor{\widetilde{M}}{\Ker \ r} \to M.
\]
This is, in fact, the smallest cover of $M$ on which $\theta$ is exact. Then $\omega_0$ descends to a symplectic form on $\overline{M}$ and the deck group $\faktor{\pi_1(M)}{\Ker \ r}$ \ acts on it by homotheties.

\bigskip

We will discuss the compatibility of our reduction with an existing complex structure. This means considering a less widespread but better studied subclass of LCS manifolds, called locally conformally \K:
\begin{definition}
\label{def:LCK}
An LCS manifold $(M, \omega, \theta)$ is called \textit{locally conformally \K\ (LCK)} if it has a complex structure $J$ and a $J$-invariant Riemannian metric $g$ compatible with $\omega$ \ie
\[
\omega (\cdot, J \cdot) = g(\cdot, \cdot).
\]
\end{definition}

\begin{definition}
\label{def:thetasharp&omega}
On an LCS manifold $(M, \omega, \theta$):
\begin{enumerate}[1)]
	\item The vector field $\theta^\omega$, the $\omega$-dual of $\theta$, is called the \textit{anti-Lee field} or the \textit{s-Lee field}.
	\item If the manifold is LCK, the vector field $\theta^\sharp$, the $g$-dual of $\theta$, is called the \textit{Lee field}.
\end{enumerate}
\end{definition}

By applying the definitions, one can see that the relation between them is
\begin{equation}
\label{eq:thetasharp&omega}
J \theta^\sharp = \theta^\omega.
\end{equation}

\bigskip

In this paper, for  contact manifolds we adopt the following:
\begin{definition}
	A manifold $M^{2k+1}$ is called a \textit{contact manifold} if it has a one-form $\alpha$ such that $\alpha \wedge d\alpha^k \neq 0$.
\end{definition}

\smallskip

This is what some authors (\eg \cite{wi}) call a \textit{co-oriented contact manifold}. For an in depth look in the theory of contact manifolds, see, for example, \cite{ge2}.

Contact and LCS manifolds are connected in a natural way (see \eg \cite{va}):

\begin{proposition}
\label{prop:lcsCont}
Let $M^{2k+1}$ be a smooth manifold with a $1$-form $\alpha$. 

Then $\alpha$ is a contact form if and only if $d_{p_1^*\theta} (p_2^*\alpha)$ is an LCS form on $S^1 \times M$, with Lee form $p_1^* \theta$ (where the Lee form $\theta$ is the angular form on $S^1$ and $p_1$ and $p_2$ are the projections on the corresponding factors).

\begin{proof}
As there is no risk of confusion, we neglect the projection mappings.

The form $d_\theta \alpha$ is already $d_\theta$-closed. Notice that 
\[
(d_\theta \alpha)^{k+1} = (d\alpha - \theta \wedge \alpha)^{k + 1} = -(k + 1)\theta \wedge \alpha \wedge d\alpha,
\]
thus 
\[
(d_\theta \alpha)^{k+1} \neq 0 \text{ if and only if } \alpha \wedge d\alpha \neq 0.
\]
\end{proof}

\end{proposition}

\section{LCS reduction for a Lie group action}
\label{sectionLcsRed}

\subsection{Reduction {\em via} foliations} In order to adapt the symplectic reduction to the LCS case, we make use of a result by Domitrz  that gives a way to construct new LCS manifolds {\em via} a reduction of any submanifold, without the existence of a group action, if certain conditions are met:

\begin{theorem}\textnormal{(\cite{dom})}
	\label{th:dom}
	Let $Q$ be a submanifold of a locally conformally symplectic manifold $(M, \omega, \theta)$ and let $\mathcal{F}$ be the foliation $TQ \cap (TQ)^\omega$. Assume $\mathcal{F}$ has constant dimension.
	If $N = \faktor{Q}{\mathcal{F}}$ is a manifold of dimension greater than $2$ and the canonical projection $\pi: Q \to N$ is a submersion, then there exists a locally conformally symplectic form $\eta$ on $N$ and a smooth function $f$ on $Q$ such that 
	\[
	\pi^* \eta = e^f \omega_{|Q}
	\]
	if and only if $\theta$ is exact on the foliation $\mathcal{F}$, that is if there exists a smooth function $h$ on $Q$ such that $\theta_{|\mathcal{F}} = dh_{|\mathcal{F}}$.
	
	Moreover, we can take $f = h$; in particular, $f = 0$ if $\theta_{|\mathcal{F}} = 0$.
\end{theorem}

\begin{remark}
\label{rem:dimFol}
As $d_\theta \omega = 0$, the distribution $\mathcal{F} = TQ \cap (TQ)^\omega$ is always involutive (\ie $[X, Y] \in \mathcal{F}$, for $X,Y\in\mathcal{F}$), for any submanifold $Q \subset (M, \omega, \theta)$. We shall discuss the constant dimension condition for our particular types of submanifolds in Subsection \ref{subsec:const_rk}.
\end{remark}

\bigskip

\subsection{Reduction {\em via} a Lie group action} Now take $(M, \omega, \theta)$ a connected LCS manifold and let $G$ be a connected Lie group acting on it. In the LCS setting, the analogous notions of Hamiltonian actions and momentum mapping from the symplectic case (except for minor sign differences in some cases, the same used by \cite{va}, \cite{hr}, \cite{bgp} and \cite{nico}, among others) are the following:

\begin{definition}
	The action of $G$ on $M$ is called \textit{twisted symplectic} if $G$ preserves the conformal LCS structure of $(\omega, \theta)$ i.e. 
	\[
	g^* \omega = e^{\varphi_g} \omega, \ \text{for some } \varphi_g \in C^\infty(M), \text{ for all } g \in G,
	\]
	and, correspondingly,
	\[
	g^* \theta = \theta + d\varphi_g.
	\]
	(note that, if $\dim(M) \ge 4$, the latter relation follows from the former).
	
\end{definition}

\bigskip

\begin{remark}
\label{rem:compvarphi}
The conformal factors have the composition rule
\[
\varphi_{gh} = \varphi_g \circ h + \varphi_h, \text{ for all } g, h \in G.
\]
\end{remark}

\begin{remark}
\label{rem:invarform}
If $G$ is compact, we can find a $G$-invariant form in the conformal class of $\omega$. Indeed, notice that 
\[
\omega_0 = \int_{G} g^* \omega \ dg
\]
is  $G$-invariant, where $dg$ is the Haar measure on $G$. Moreover,
\[
\omega_0 = \int_{G} g^* \omega \ dg = \big(\int_{G} e^{\varphi_g} dg\big) \omega,
\]
and so is an LCS form in the same conformal class as $\omega$. For brevity, denote
\[
F = \ln \int_{G} e^{\varphi_g} dg,
\]
so that $\omega_0 = e^F \omega$ has the Lee form $\theta_0 = \theta + dF$. This function has an interesting behavior with respect to the $G$-action: for an $h \in G$, we have
\begin{equation}
\begin{split}
\label{eq:FbehaviorG}
h^* F &= F \circ h = \ln \int_{G} e^{\varphi_g \circ h} dg = \ln \int_{G} e^{\varphi_{gh} - \varphi_h} dg = \ln \int_{G} e^{\varphi_{gh}} dg - \varphi_h \\
&= F - \varphi_h.
\end{split}
\end{equation}
(we used here \ref{rem:compvarphi}). We shall use this fact later.
\end{remark}

\bigskip

Introduce the ``twisted Lie derivative'' operator $\mathcal{L}_X^\theta$:
\[
\mathcal{L}_X^\theta \alpha = \mathcal{L}_X \alpha - \theta(X) \alpha,
\]
for any differential form $\alpha$ on $M$. One can see that this satisfies an equivalent of the Cartan formula:
\begin{equation}
	\label{eq:cartan}
	\mathcal{L}_X^\theta = d_\theta i_X + i_X d_\theta.
\end{equation}
For any $a \in \mathfrak{g}$, denote by $X_a \in \mathcal{X}(M)$ the corresponding fundamental vector field. Then
\begin{equation}
\begin{split}
\label{eq:liederiv}
	\mathcal{L}_{X_a}^\theta \omega
	 &= \mathcal{L}_{X_a} \omega - \theta(X_a) \omega\\
	 & = \lim\limits_{t \to 0} \frac{(\exp(ta))^* \omega - \omega}{t} - \theta(X_a) \omega  \\
	&= \lim\limits_{t \to 0} \frac{e^{\varphi_{\exp(ta)}} \omega - \omega}{t} - \theta(X_a) \omega\\
	& = \left(\frac{d}{dt}\restrict{|t=0}(e^{\varphi_{\exp(ta)}}) -\theta(X_a)\right)\omega.
	\end{split}
\end{equation}

\begin{remark} Observe  that 
\[
\mathcal{L}_{X_a}^\theta \omega = f \omega,
\]
for a smooth function $f$. However, using the Cartan formula  (\ref{eq:cartan}) and applying $d_\theta$, we see that $f$ must be constant, so $\mathcal{L}_{X_a}^\theta \omega = c_a \omega$ for some $c_a \in \R$, for each $a \in \mathfrak{g}$.
Again from $(\ref{eq:cartan})$, note that
\[
d_\theta (i_{X_a} \omega) = c_a \omega,
\]
so the dual of $X_a$ {\em via} $\omega$ is not $d_\theta$ - closed, an important difference from the case of symplectic actions.
\end{remark}

\bigskip
	
\begin{definition}
The action is called \textit{twisted Hamiltonian} if $i_{X_a}\omega$ is $d_\theta$ - exact, for all $a \in \mathfrak{g}$.	
\end{definition}

\bigskip

\begin{remark}
For later use, note that for a twisted Hamiltonian action, (\ref{eq:liederiv}) implies  
\begin{equation}
\label{eq:liederiv2}
\frac{d}{dt}_{|t=0}(e^{\varphi_{\exp(ta)}}) = \theta(X_a),\ \text{for all}\ \  a \in \mathfrak{g}.
\end{equation}
\end{remark}

Assume from now on that $\theta$ is not exact (\ie $M$ is not GCS).
Then $H^0_\theta(M) = 0$ (see \ref{prop:H0MN}), so there is a unique choice of $\rho_a \in C^\infty (M)$ such that
\[
d_\theta \rho_a = i_{X_a} \omega.
\]

\begin{definition}
The LCS \textit{momentum mapping} $\mu: M \to \mathfrak{g}^*$ associated to a twisted Hamiltonian action is 
\[
\mu(x)(a) = \rho_a (x),\ \  \text{for all}\ \ a \in \mathfrak{g}.
\]
\end{definition}

\begin{remark}
\label{rem:changemu}
If an action of $G$ on $(M, \omega, \theta)$ is twisted Hamiltonian with momentum mapping $\mu$, then the same action on $(M, e^f \omega, \theta + df)$ is still twisted Hamiltonian with momentum mapping $e^f \mu$.
\end{remark}

\begin{lemma}
\label{lem:echivar}
In the conditions above:
\begin{enumerate}[1)]
	\item \label{lem:echivar1} The map $\rho : \mathfrak{g} \to \mathcal{C}^\infty(M)$ is a Lie algebra homomorphism, with respect to the twisted Poisson bracket on $\mathcal{C}^\infty(M)$, defined as:
	\[
	\{f, g\}_\theta = \omega((d_\theta f)^\omega, (d_\theta g)^\omega),
	\]
	where, for a $1$-form $\alpha$, $\alpha^\omega$ is the $\omega$-dual of $\alpha$.
	
	\item \label{lem:echivar2} If $\omega$ is $G$-invariant, the momentum mapping $\mu: M \to \mathfrak{g}^*$ is equivariant with respect to the coadjoint action on $\mathfrak{g}^*$.
	
	\item \label{lem:echivar3} If $G$ is compact, $\mu$ satisfies a twisted equivariance:
	\[
	\mu(g \cdot x) = e^{\varphi_g} g \cdot \mu(x),\ \ \text{for all}\ \ g \in G, x \in M
	\] 
	(see \cite{hr} for equivalent formulations).
\end{enumerate}
\end{lemma}

\begin{proof}
	\ref{lem:echivar1}) Since for any $a \in \mathfrak{g}, \rho_a$ is uniquely determined, we only need to prove that $\{ \rho_a, \rho_b \}_\theta$ satisfies the required condition for $\rho_{[a,b]}$ \ie 
	\[
	d_\theta \{ \rho_a, \rho_b \}_\theta = i_{X_{[a,b]}} \omega.
	\]
	Indeed,
	\begin{align*}
	d_\theta \{ \rho_a, \rho_b \}_\theta &= d_\theta (\omega((d_\theta \rho_a)^\omega, (d_\theta \rho_b)^\omega)) = d_\theta (\omega(X_a, X_b)) \\
	&= d_\theta i_{X_b} i_{X_a} \omega = \mathcal{L}^\theta_{X_b} i_{X_a} \omega - i_{X_b} d_\theta i_{X_a} \omega.
	\end{align*}
	As $[\mathcal{L}^\theta_X, i_Y] = i_{[X,Y]}$ and  the action is twisted Hamiltonian, so $\mathcal{L}^\theta_{X_a} \omega = \mathcal{L}^\theta_{X_b} \omega = 0$, we have:
	\begin{align*}
	\mathcal{L}^\theta_{X_b} i_{X_a} \omega - i_{X_b} d_\theta i_{X_a} \omega &= 
	i_{[X_b, X_a]} \omega + i_{X_a} \mathcal{L}^\theta_{X_b} \omega - i_{X_b} (\mathcal{L}^\theta_{X_a} \omega - i_{X_a}d_\theta \omega) \\
	&=i_{[X_b, X_a]} \omega = i_{X_{[a,b]}} \omega,
	\end{align*}
which proves the claim.
	
	\ref{lem:echivar2}) Recall the definition of the coadjoint action of $G$ on $\mathfrak{g}^*$: $\Ad^*(g)(\xi)(a) = \xi(\Ad(g^{-1})a)$. The condition that $\mu$ is equivariant is then equivalent to
	\begin{equation}
	\label{eq:echivar1}
	\rho_a(g \cdot x) = \rho_{\Ad(g^{-1})a}(x), \text{ for all } x \in M, \  a \in \mathfrak{g}.
	\end{equation}
	
	We mimic the  proof of the analogue fact in the symplectic case (see \cite[pp. 119-120]{br}).
	
	Since $G$ is connected and both sides of (\ref{eq:echivar1}) are $G$-actions, we only need to prove the formula for $g = \exp(tb)$ for any $b \in \mathfrak{g}$ and $t \in \R$ \ie 
	\[
	\rho_a(\exp(tb) \cdot x) = \rho_{\Ad(\exp(-tb))a}(x).
	\]
	Moreover, since this is true for $t = 0$, it is enough to prove that the derivatives of the two sides are equal.
	
	For the left hand side,
	\begin{align*}
	\frac{d}{dt} \rho_a(\exp(tb) \cdot x) &= d \rho_a (X_b(\exp(tb)\cdot x))\\
	& = \omega(X_a(\exp(tb)\cdot x) , X_b(\exp(tb)\cdot x)) \\ 
	& + \theta(X_b)(\exp(tb)\cdot x) \cdot \rho_a(\exp(tb)\cdot x) \\
	&= \omega ( X_{\Ad(exp(-tb))a}(x), X_{\Ad(exp(-tb))b}(x) ) \\
	&= \omega ( X_{\Ad(exp(-tb))a}(x), X_b(x) )
	\end{align*}
	where, on the third equality, we have used (\ref{eq:liederiv2}) to get $\theta(X_b) = 0$, the identity
	\[
	g_* (X_a(x)) = X_{\Ad(g)a}(g \cdot x)
	\]
	and that $\omega$ is $G$-invariant.
	
	As for the right hand side, using point \ref{lem:echivar1}), we have
	\begin{align*}
	\frac{d}{dt} \rho_{\Ad(\exp(-tb))a}(x) &= \rho_{[-b, \Ad(exp(-tb))a]}(x) = - \{ \rho_b, \rho_{\Ad(\exp(-tb))a} \}_\theta (x) \\ 
	&= \omega(X_{\Ad(\exp(-tb))a}(x), X_b(x)).
	\end{align*}
	
	\ref{lem:echivar3}) Take $\omega_0$ the $G$-invariant LCS form in the conformal class of $\omega$, as described in \ref{rem:invarform} and take
	\[
	F = \ln \int_{G} e^{\varphi_g} dg,
	\]
	so $\omega_0 = e^F \omega$. According to point \ref{lem:echivar2}), its momentum mapping, denoted by $\mu_0$, is equivariant; from \ref{rem:changemu}, $\mu_0 = e^F \mu$.
	
	Then
	\begin{equation*}
	\begin{split}
	\mu(g \cdot x) &= e^{-F(g \cdot x)} \mu_0 (g \cdot x)\\
	& = e^{-F(x) + \varphi_g(x)} g \cdot \mu_0(x) \ \ \text{by \eqref{eq:FbehaviorG}}\\
	&= e^{\varphi_g(x)} g \cdot \mu(x).
	\end{split}
	\end{equation*} 
\end{proof}

\subsection{Constancy of the rank of the characteristic distribution } \label{subsec:const_rk} 

Our aim is that the LCS reduction be a reduction along a foliation on a level set of $\mu$ (as it is in the symplectic case). Let $\xi$ be a regular value of $\mu$. With \ref{th:dom} in mind, we study the behavior of $\theta$ on the distribution $\mathcal{F} = T\mu^{-1}(\xi) \cap (T\mu^{-1}(\xi))^\omega$, which we call {\em characteristic distribution}. We already know it is involutive (see \ref{rem:dimFol}); next, we find sufficient conditions for it to have constant rank.

\bigskip

Take $x \in M$ and $v \in T_x\mu^{-1}(\xi) = \Ker\, d\mu$. Then, for all $a \in \mathfrak{g}$,
\begin{align*}
0 &= d_x \mu(v)(a) = d_x \rho_a(v)\\
& = (i_{X_a}\omega) (v) + \theta(v) \mu(x)(a) \\
&= \omega (X_a, v) + \theta(v) \xi(a)\\
& = \omega(X_a + \theta^\omega \xi(a), v),
\end{align*}
where $\theta^\omega$ is the s-Lee field (see \eqref{eq:thetasharp&omega}). This means that 
\begin{equation}
\label{eq:ortho}
(T\mu^{-1}(\xi))^\omega = \{ X_a +  \xi(a) \theta^\omega \ | \ a \in \mathfrak{g} \}.
\end{equation}

\begin{remark}
Note the difference with the symplectic case, where the symplectic dual of the level sets was the $G$-orbit. In fact, this intersects the tangent space of the orbit along a hyperplane, specifically the kernel of $\xi$.
\end{remark}

\begin{lemma}
\label{lem:thetaG}
$G$ acts on $\theta^\omega$ \textit{via} the formula
\[
g_* \theta^\omega = e^{-\varphi_{g^{-1}}} \theta^\omega - d(e^{-\varphi_{g^{-1}}})^\omega
\]

\begin{proof}
Indeed, for an $x \in M$ and $X \in \mathcal{X}(M)$,
\begin{equation*}
\begin{split}
\omega_x (g_* (\theta_{g^{-1}x}^\omega), X_x) &= \omega_x (g_* (\theta_{g^{-1}x}^\omega), g_* g^{-1}_* X_x) = (g^* \omega)_{g^{-1}x} (\theta^\omega_{g^{-1}x}, g^{-1}_* X_x) \\
&= e^{\varphi_g (g^{-1}x)} \omega_{g^{-1}x} (\theta^\omega_{g^{-1}x}, g^{-1}_* X_x) \\ 
&= e^{-\varphi_{g^{-1}}(x)} \theta_{g^{-1}x} (g^{-1}_* X_x) = e^{-\varphi_{g^{-1}}(x)} ((g^{-1})^* \theta)_x (X_x) \\
&= e^{-\varphi_{g^{-1}}(x)} (\theta + d \varphi_{g^{-1}})(X_x) \\
&= e^{-\varphi_{g^{-1}}(x)} \theta - d(e^{-\varphi_{g^{-1}}})(X_x).
\end{split}
\end{equation*}
\end{proof}
\end{lemma}

\begin{lemma}
\label{lem:dimconst}
Let $(M, \omega, \theta)$ be a connected LCS manifold with a twisted Hamiltonian action of the connected Lie group $G$ and $\xi$ a regular value of the momentum mapping $\mu$. Assume that one of the following conditions is met:
\begin{itemize}
	\item The action of $G$ preserves the LCS form $\omega$.
	\item $G$ is compact and $\xi \wedge \theta_x(X_\cdot) = 0$ for all $x \in \mu^{-1}(\xi)$.
\end{itemize} 
Then the distribution $\mathcal{F} = T\mu^{-1}(\xi) \cap (T\mu^{-1}(\xi))^\omega$ is of constant rank.
\end{lemma}

\begin{proof}
Take an $x \in \mu^{-1}(\xi)$. Since $\xi$ is a regular value, $\dim T_x \mu^{-1}(\xi) = \dim M - \dim G$, so 
\[
\dim (T_x \mu^{-1}(\xi))^\omega = \dim G.
\]
This means that the linear mapping
\begin{equation}
\label{eq:linearPSI}
\psi: \mathfrak{g} \to (T_x \mu^{-1}(\xi))^\omega, \ \psi(a) = X_a +  \xi(a) \theta^\omega
\end{equation}
is an isomorphism. 

Since $\omega$ is $G$-invariant or $G$ is compact, recall from \ref{lem:echivar} that $\mu$ is twisted equivariant (this defaults to equivariancy in the first case):
\[
\mu(g \cdot x) = e^{\varphi_g} g \cdot \mu(x),\ \ \text{for all}\ \ g \in G,\ x \in M.
\]
An element $v \in T_x\mu^{-1}(\xi) \cap (T_x\mu^{-1}(\xi))^\omega$ then satisfies:
\begin{equation}
\label{eq:calculg_xi}
\begin{split}
0 &= d_x\mu(v) = d_x\mu( X_a +  \xi(a) \theta^\omega)\\
& = \frac{d}{dt}\restrict{t=0} ( \mu (\exp(ta) \cdot x) ) + \xi(a) d_x\mu(\theta^\omega) \\
&= \frac{d}{dt}\restrict{t=0} (e^{\varphi_{\exp(ta)}} \exp(ta) \cdot \mu (x) ) + \xi(a) d_x\mu(\theta^\omega) \\
&= \frac{d}{dt}\restrict{t=0} (e^{\varphi_{\exp(ta)}} ) \xi + \frac{d}{dt}_{t=0} (\exp(ta) \cdot \xi ) + \xi(a) d_x\mu(\theta^\omega) \\
&= \theta(X_a) \xi + (\Ad^*)_*(a)(\xi) + \xi(a)d_x\mu(\theta^\omega) \\
&= (\Ad^*)_*(a)(\xi) + (\xi \wedge \theta_x(X_\cdot))(\cdot, a).
\end{split}
\end{equation}
For the last equality, we have used that, for any $b \in \mathfrak{g}$,
\begin{equation*}
d_x\mu(\theta^\omega)(b) = d_x \rho_b(\theta^\omega) = \omega_x(X_b, \theta^\omega) + \rho_b(x) \theta_x(\theta^\omega) = -\theta_x(X_b).
\end{equation*}

In both cases, our hypotheses imply that $\xi \wedge \theta_x(X_\cdot) = 0$ (see \ref{eq:liederiv2}), so \eqref{eq:calculg_xi} gives a characterization of the elements of the foliation $\mathcal{F}$:
\[
v = X_a +  \xi(a) \theta^\omega \in T_x\mu^{-1}(\xi) \cap (T_x\mu^{-1}(\xi))^\omega = \mathcal{F}_x \ \ \text{if and only if}\ \  a \in \mathfrak{g}_\xi,
\]
where $\mathfrak{g}_\xi$ is the Lie algebra of $G_\xi$. In particular,
\[
\mathcal{F}_x \iso g_\xi,
\]
\textit{via} the restriction of $\psi$, which guarantees the dimension of $\mathcal{F}$ is constant on $\mu^{-1}(\xi)$. 
\end{proof}

\bigskip

\subsection{The reduction theorem} We now have all the facts needed to state and prove our main result:

\begin{theorem}
\label{th:LCSred}
Let $(M, \omega, \theta)$ be a connected LCS manifold and $G$ a connected Lie group acting twisted Hamiltonian on it.

Let $\mu$ be the momentum mapping and $\xi \in \mathfrak{g}^*$ a regular value. Denote by $\mathcal{F} = T\mu^{-1}(\xi) \cap (T\mu^{-1}(\xi))^\omega$. Assume that one of the following conditions is met:
\begin{itemize}
	\item The action of $G$ preserves the LCS form $\omega$.
	\item $G$ is compact, $\xi \wedge \theta_x(X_\cdot) = 0$ for all $x \in \mu^{-1}(\xi)$ and there exists a function $h$ on $\mu^{-1}(\xi)$ such that $\theta_{| \mathcal{F}} = dh$.
\end{itemize}

If $N_\xi := \faktor{\mu^{-1}(\xi)}{\mathcal{F}}$ is a smooth manifold and $\pi : \mu^{-1}(\xi) \to N_\xi$ is a submersion, then $N_\xi$ has an LCS structure such that the LCS form $\omega_\xi$ satisfies
\[
\pi^* \omega_\xi = e^f \omega_{|\mu^{-1}(\xi)}
\]
for some $f \in C^\infty(\mu^{-1}(\xi))$. 

Moreover, one can take $f = h$; in particular, $f = 0$ if the action preserves the LCS form.

\begin{proof}

From \ref{lem:dimconst} we obtain that, in either case, $\mathcal{F}$ has constant rank and is thus a foliation. In the second case, with the assumption that $\theta$ is exact along $\mathcal{F}$, we are already fully in the conditions of \ref{th:dom}, and the proof is concluded.

{\bf If $\omega$ is $G$-invariant.} According to \ref{th:dom}, it is sufficient to prove that $\theta$ is exact on elements of the form described by equality (\ref{eq:ortho}). In fact, we will show it is zero. For this, recall  formulae (\ref{eq:liederiv}) and (\ref{eq:liederiv2}). Since $\theta$ is also $G$-invariant, we have 
\[
\theta(X_a) = 0, \ \ \text{for all}\ \ a \in \mathfrak{g},
\]
thus
\begin{equation}
\label{eq:Thredlcs1}
\theta(X_a + \xi(a) \theta^\omega ) = 0,
\end{equation}
and the conclusion follows immediately, allowing us to choose $f = 0$.
\end{proof}
\end{theorem}

\begin{remark}
One can see that, if the manifold on which $G$ acts is symplectic and the action admits a symplectic momentum mapping, our procedure gives precisely the classical reduction scheme. 
\end{remark} 

\begin{remark}(Reduction at zero)
\label{rem:redzero}
If $0$ is a regular value of the momentum mapping, then the reduction at zero does not depend on the choice of form in the LCS structure. This is in sharp contrast with the non-zero case, where the level sets themselves can vary wildly. One can also see that in this case, factoring by the given foliation is the same as factoring by $G$ itself.
\end{remark}

\bigskip

The factorization via the foliation $\mathcal{F}$ can also be seen to be a quotient via a group action; the difference from the symplectic case is that, in most cases, this will be a quotient via the action of the universal cover of the stabilizer $G_\xi$:
\begin{proposition}
\label{prop:actiuneGrup}
There is an action of $\tilde{G}_\xi$ on $\mu^{-1}(\xi)$ such that, under the conditions of \ref{th:LCSred}, 
\[
N_\xi = \faktor{\mu^{-1}(\xi)}{\tilde{G}_\xi}.
\]

\begin{proof}
Consider the map
\[
\psi: \mathfrak{g}_\xi \to \mathcal{F} \subset \mathcal{X}(\mu^{-1}(\xi)), \ \psi(a) = X_a + \xi(a)\theta^\omega.
\]
This can be seen to be a Lie algebra anti-isomorphism. Indeed,  
\begin{equation}
\begin{split}
\label{eq:actiuneGrup1}
[\psi(a), \psi(b)] &= [X_a + \xi(a)\theta^\omega, X_b + \xi(b)\theta^\omega] \\
&= [X_a, X_b] + [\xi(b)X_a - \xi(a) X_b, \theta^\omega] \\
&= -X_{[a,b]} + \mathcal{L}_{\xi(b)X_a - \xi(a) X_b} \theta^\omega.
\end{split}
\end{equation}
We will calculate the last term separately, using \ref{lem:thetaG} and \eqref{eq:liederiv2}:
\begin{equation*}
\begin{split}
\mathcal{L}_{\xi(b)X_a - \xi(a) X_b} \theta^\omega &= \xi(b) \lim\limits_{t \to 0} \frac{(e^{ta})_* \theta^\omega - \theta^\omega}{t} - \xi(a) \lim\limits_{t \to 0} \frac{(e^{tb})_* \theta^\omega - \theta^\omega}{t} \\
&= \xi(b) \lim\limits_{t \to 0} \frac{e^{-\varphi_{exp(-ta)}} \theta^\omega - d(e^{-\varphi_{exp(-ta)}})^\omega - \theta^\omega}{t} \\
&- \xi(a) \lim\limits_{t \to 0} \frac{e^{-\varphi_{exp(-tb)}} \theta^\omega - d(e^{-\varphi_{exp(-tb)}})^\omega - \theta^\omega}{t} \\
&= \lim\limits_{t \to 0} \frac{\xi(b)(e^{-\varphi_{exp(-ta)}} - 1)\theta^\omega - \xi(a)(e^{-\varphi_{exp(-tb)}} - 1)\theta^\omega}{t} \\
&+ \lim\limits_{t \to 0} \frac{\xi(a)d(e^{-\varphi_{exp(-tb)}})^\omega - \xi(b)d(e^{-\varphi_{exp(-ta)}})^\omega}{t} \\
&= (\xi(b) \theta(X_a) - \xi(a) \theta(X_b)) \theta^\omega \\
&+ \xi(a) \frac{d}{dt}_{|t=0} d(e^{-\varphi_{exp(-tb)}})^\omega - \xi(b) \frac{d}{dt}_{|t=0} d(e^{-\varphi_{exp(-ta)}})^\omega \\
&= (\xi(b) \theta(X_a) - \xi(a) \theta(X_b)) \theta^\omega \\
&+ \xi(a) (d(\frac{d}{dt}_{|t=0} e^{-\varphi_{exp(-tb)}}))^\omega - \xi(b) (d(\frac{d}{dt}_{|t=0} e^{-\varphi_{exp(-ta)}}))^\omega \\
&= (\xi(b) \theta(X_a) - \xi(a) \theta(X_b)) \theta^\omega + \xi(a) d(\theta(X_b))^\omega - \xi(b) d(\theta(X_a))^\omega \\
&= (\xi(b) \theta(X_a) - \xi(a) \theta(X_b)) \theta^\omega + d(\xi(a) \theta(X_b) - \xi(b) \theta(X_a))^\omega \\
&= 0,
\end{split}
\end{equation*}
the last equality following from a hypothesis of the reduction theorem. Coming back to \eqref{eq:actiuneGrup1}, we have
\begin{equation}
\label{eq:actiuneGrup4}
[\psi(a), \psi(b)] = -X_{[a,b]}.
\end{equation}

On the other hand,
\begin{equation}
\label{eq:actiuneGrup2}
\psi([a,b]) = X_{[a,b]} + \xi([a,b])\theta^\omega;
\end{equation}
define extension by left invariance for $\xi, a$ and $b$ on $G$:
\begin{equation*}
\begin{split}
\alpha \in \Omega^1(G), \ \alpha_g &= (L_{g^{-1}})^* \xi, \\
Y \in \mathcal{X}(G), \ Y_g &= (L_g)_* a, \ \ \ \ \forall g \in G. \\
Z \in \mathcal{X}(G), \ Z_g &= (L_g)_* b,
\end{split}
\end{equation*}
Note that the flows of $Y$ and $Z$ are $\phi^X_t(g) = g \cdot e^{ta}$ and $\phi^Y_t(g) = g \cdot e^{tb}$, respectively. Notice that, because $a \in g_\xi$, $\alpha$ is invariant under the flow of $X$:
\begin{equation*}
\begin{split}
((\phi^X_t)^*\alpha)_g (v) &= \alpha_{g \cdot e^{ta}} ((\phi^X_t)_* v) = \alpha_{g \cdot e^{ta}} ((R_{e^{ta}})_* v) \\
&= \xi ((L_{e^{-ta}})_* (L_{g^{-1}})_* (R_{e^{ta}})_* v) \\
&= (e^{ta} \cdot \xi) ((L_{g^{-1}})_* v) = \xi ((L_{g^{-1}})_* v) \\
&= \alpha_g (v).
\end{split}
\end{equation*}

Then
\begin{equation}
\begin{split}
\label{eq:actiuneGrup3}
\xi([a,b]) &= \alpha_e (\mathcal{L}_X Y) = \alpha_e (\frac{d}{dt}_{|t=0} (\phi^X_{-t})_* Y_{e^{ta}}) = \frac{d}{dt}_{|t=0} \alpha_e ( (\phi^X_{-t})_* Y_{e^{ta}}) \\
&= \frac{d}{dt}_{|t=0} ((\phi^X_{-t})^*\alpha)_{e^{ta}}(Y_{e^{ta}}) = \frac{d}{dt}_{|t=0} (\alpha_{e^{ta}}(Y_{e^{ta}})) \\
&= \frac{d}{dt}_{|t=0} \xi(b) = 0.
\end{split}
\end{equation}
The conclusion follows from \eqref{eq:actiuneGrup4}, \eqref{eq:actiuneGrup2} and \eqref{eq:actiuneGrup3}.

Using the Lie-Palais Theorem (\cite[p. 1047]{mil}), it follows that there is an action of $\tilde{G}_\xi$ on $\mu^{-1}(\xi)$ with $\psi$ as its differential. Thus,
\[
\faktor{\mu^{-1}(\xi)}{\mathcal{F}} = \faktor{\mu^{-1}(\xi)}{\tilde{G}_\xi}.
\]

\end{proof}
\end{proposition}

\subsection{Cobordism class of the quotient} A natural question to ask, since the reduction described above produces a great number of LCS manifolds (either by varying the regular value or the LCS form in its conformal class), is whether there is any connection between them. This is what we study for the rest of this section. 

\begin{proposition}
\label{prop:cobordVal}
Let $(M, \omega, \theta)$ be a connected LCS manifold and $G$ a compact Lie group acting twisted Hamiltonian on it. Let $\mu$ be the momentum mapping and 
\[
\gamma :[0, 1] \to \mathfrak{g}^*
\]
a smooth path from $\xi_0$ to $\xi_1$. Assume that the hypotheses of \ref{th:LCSred} are met for all $\gamma(t), \ 0 \le t \le 1$ and, furthermore, that $\dim G_{\gamma(t)}$ is constant.

Then the LCS manifolds $N_{\xi_0}$ and $N_{\xi_1}$ are cobordant. 
\end{proposition}

\begin{proof}
	Consider 
	\[
	S:= \mu^{-1}(\Im \gamma).
	\]
Then $S$ is	a submanifold with boundary of $M$, since all $\gamma(t)$ are regular values. More precisely, 
	\[
	\partial S = \mu^{-1}(\xi_0) \cup \mu^{-1}(\xi_1),
	\]
	so it is a cobordism between the two level sets. Define the foliation $\mathcal{F}$ on $S$:
	\[
	\mathcal{F}_x = T_x\mu^{-1}(\mu(x)) \cap (T_x\mu^{-1}(\mu(x)))^\omega, x \in S.
	\]
	Notice that $\mathcal{F}$ contains all the foliations that are the object of the previous theorem, for all regular values $\gamma(t)$. In particular, it is tangent to the level sets of $\mu$, has constant dimension by assumption and its leaves are the leaves of all those foliations.
	
	Since all $\gamma(t)$ satisfy the hypotheses of \ref{th:LCSred}, the set
	\[
	N := \faktor{S}{\mathcal{F}}\, (= \bigcup\limits_{0 \le t \le 1} N_{\gamma(t)})
	\]
	is a smooth manifold with boundary $\partial N = N_{\xi_0} \cup N_{\xi_1}$ and so is the cobordism we wanted. 
\end{proof}

\bigskip

Something similar happens when we vary the LCS form $\omega$ in its conformal class. For this, recall \ref{rem:changemu}: if a twisted Hamiltonian $G$-action on $(M, \omega, \theta)$ has the momentum mapping $\mu$, then the same  $G$-action on $(M, e^f \omega, \theta + df)$ is still twisted Hamiltonian and has the momentum mapping $e^f \mu$, which in turn has wholly different level sets.

\begin{proposition}
\label{prop:cobordFct}
Let $(M, \omega, \theta)$ be a connected LCS manifold and $G$ a compact Lie group acting twisted Hamiltonian on it. Let $\mu$ be the momentum mapping and 
\[
f :[0, 1] \to \mathcal{C}^\infty(M)
\]
a smooth path. Let $\xi \in \mathfrak{g}^*$ and assume that the hypotheses of \ref{th:LCSred} are met for all $G$-actions on $(M, e^{f_t}\omega, \theta + df_t)$ with regard to the fixed value $\xi$.

Denote by $\mu_t$ the momentum mapping corresponding to the LCS form $e^{f_t} \omega$ (so $\mu_t = e^{f_t} \mu$) and by $N_\xi^t$ the LCS manifold given by \ref{th:LCSred} applied to each such form.
	
Then the LCS manifolds $N_\xi^0$ and $N_\xi^1$ are cobordant.
\end{proposition}

\begin{proof}
	Consider the map
	\[
	H:M \times [0,1] \to \mathfrak{g}^*,\ \  H(\cdot, t) = \mu_t.
	\]
	Since $\xi$ is a regular value for each $\mu_t$, the set 
	\[
	S := H^{-1}(\xi)
	\] is a submanifold with boundary of $M \times [0,1]$, 
	\[
	\partial S = \mu_0^{-1}(\xi) \cup \mu_1^{-1}(\xi).
	\]
	Again define the foliation $\mathcal{F}$ on $S$:
	\[
	\mathcal{F}_{x, t} = T_{(x,t)}\mu_t^{-1}(\xi) \cap T_{(x,t)}\mu_t^{-1}(\xi)^{e^{f_t}\omega},\ \  x \in S.
	\]
	Notice that $\mathcal{F}$ contains all the foliations that are the object of the previous theorem, for all $G$-actions on $(M, e^{f_t}\omega)$. In particular, it is tangent to the level sets of the $\mu_t$-s, has constant dimension equal to $\dim G_\xi$ and and its leaves are the leaves of all those foliations. 
	
	Since all LCS forms $e^{f_t} \omega$ satisfy the hypotheses of \ref{th:LCSred}, the set
	\[
	N := \faktor{S}{\mathcal{F}}\ ( = \bigcup\limits_{0 \le t \le 1} N_\xi^t)
	\]
	is a smooth manifold with boundary $\partial N = N_\xi^0 \cup N_\xi^1$ and so is the desired cobordism. 
\end{proof} 

\section{Compatibility with a complex structure}
\label{sectionLCK}

We now discuss the conditions under which our reduction is compatible with a complex structure. This is similar to \K \ reduction, except we will give the necessary conditions under which the complex structure on the manifold descends naturally to the quotient, as a `shifting trick' (see \eg \cite[pp. 135-137]{br}) is not possible in the LCK setting (because a product of LCK manifolds has no natural LCK structure).

Take $(M, J, g, \omega, \theta)$ an LCK manifold with a twisted hamiltonian action of the compact, connected Lie group $G$. Recall the notation from the previous section for the conformal factors involved:
\[
h^* \omega = e^{\varphi_h} \omega, \text{ for all } h \in G.
\]
The following proposition describes the relationship between the effect of $G$ on the metric $g$ and on the complex structure $J$:

\begin{proposition}
\label{prop:gJ} 
\begin{enumerate}[1)]
	\item \label{prop:gJi1} If $G$ acts on the metric $g$ conformally \ie 
	\[
	h^* g = e^{\psi_h} g, \text{ for all } h \in G,
	\]
	then $\varphi_h = \psi_h$ for all $h \in G$ and $J$ is $G$-invariant (i.e. the action is holomorphic).
	\item \label{prop:gJi2} If the action is holomorphic, then $h^* g = e^{\varphi_h} g$, for all $h \in G$.
\end{enumerate}

\begin{proof} Both are straightforward calculations:
	
	1) Observe that $\psi$ satisfies the same composition rule as $\varphi$, described in \ref{rem:compvarphi}; in particular, $\psi_h = -\psi_{h^{-1}} \circ h$. Take $x \in M$,  $v, w \in T_xM$ and $h \in G$. Then
	\begin{equation*}
	\begin{split}
	g_x (h_*^{-1} J h_* v, w) &= g_x (h_*^{-1} J h_* v, h_*^{-1} h_* w) = ((h^{-1})^* g)_{h(x)}(J h_* v, h_*w) \\
	&= e^{\psi_{h^{-1}}(h(x))} g_{h(x)} (J h_* v, h_*w) = -e^{-\psi_h(x)} \omega_{h(x)}(h_* v, h_*w) \\
	&= -e^{-\psi_h(x)} (h^* \omega)_x (v, w) = -e^{\varphi_h(x)-\psi_h(x)}\omega_x(v, w) \\
	&= e^{\varphi_h(x)-\psi_h(x)} g_x (Jv, w),
	\end{split}
	\end{equation*}
	so 
	\[
	h_*^{-1} J h_* = e^{\varphi_h-\psi_h} J. 
	\]
	But since the left-hand side is also a complex structure, we obtain $\varphi_h = \psi_h$ and $J$ is $G$-invariant.
	
	2) Again for $x \in M$,  $v, w \in T_xM$ and $h \in G$, we have
	\begin{equation*}
	\begin{split}
	(h^*g)_x (v, w) &= g_{h(x)}(h_* v, h_* w) = \omega_{h(x)}(h_* v, J h_* w) = \omega_{h(x)}(h_* v, h_* Jw) \\
	&= (h^* \omega)_x (v, Jw) = e^{\varphi_h(x)} \omega_x (v, Jw) \\
	&= e^{\varphi_h(x)} g (v, w),
	\end{split}
	\end{equation*}
	the conclusion we wanted.
\end{proof}
\end{proposition}

\bigskip

We can now give the reduction theorem for LCK manifolds; besides the restrictions present in the \K \ setting (see \cite[p. 137]{br}), the necessary requirement is that the s-Lee field $\theta^\omega$ be holomorphic \ie $\mathcal{L}_{\theta^\sharp} J = 0$.

\begin{theorem}
\label{thm:LCKred}
Let $(M, J, g, \omega, \theta)$ be a connected LCK manifold and $G$ a connected Lie group whose action on $M$ is twisted Hamiltonian and holomorphic.

Let $\mu$ be the momentum mapping and $\xi \in \mathfrak{g}^*$ a regular value. Denote by $\mathcal{F} = T\mu^{-1}(\xi) \cap (T\mu^{-1}(\xi))^\omega$. Assume that one of the following conditions is met:
\begin{itemize}
	\item The action of $G$ preserves the LCS form $\omega$.
	\item $G$ is compact, $\xi \wedge \theta_x(X_\cdot) = 0$ for all $x \in \mu^{-1}(\xi)$ and there exists a function $h$ on $\mu^{-1}(\xi)$ such that $\theta_{| \mathcal{F}} = dh$.
\end{itemize}

Assume further that $G_\xi = G$ and that the s-Lee field $\theta^\omega$ is holomorphic.
	
If $N_\xi := \faktor{\mu^{-1}(\xi)}{\mathcal{F}}$ is a smooth manifold and $\pi : \mu^{-1}(\xi) \to N_\xi$ is a submersion, then $N_\xi$ has an LCK structure $(J_\xi, g_\xi, \omega_\xi, \theta_\xi)$ which satisfies
\[
\pi^* \omega_\xi = e^f \omega_{|\mu^{-1}(\xi)}
\]
for some $f \in C^\infty(\mu^{-1}(\xi))$.

Moreover, one can take $f = h$; in particular, $f = 0$ if the action preserves the LCK form.

\begin{proof}
We already know from \ref{th:LCSred} that $N_\xi$ has the LCS form $\omega_\xi$ satisfying
\[
\pi^* \omega_\xi = e^f \omega_{|\mu^{-1}(\xi)}
\]
for an $f \in C^\infty(\mu^{-1}(\xi))$. 

We now show how one can define an almost complex structure on the quotient. Take $x \in \mu^{-1}(\xi)$. Recall from \ref{lem:dimconst} that, since $G_\xi = G$, we have
\[
(T_x\mu^{-1}(\xi))^\omega = \{ X_a + \xi(a) \theta^\omega \ | \ a \in \mathfrak{g} \} \subset T_x\mu^{-1}(\xi),
\]
so $\mathcal{F} = (T\mu^{-1}(\xi))^\omega$.

For brevity, denote by $A_x = (T_x\mu^{-1}(\xi))^\omega$ and take $B_x = A_x^\perp \cap T_x\mu^{-1}(\xi)$. Consequently,
\[
B_x = \{ v \in T_x M \ | \ g(v, w) = \omega(v, w) = 0, \forall w \in A_x \}.
\]
This proves that $B_x$ is $J$-invariant, since then
\begin{equation}
\label{eq:thmLCKred1}
J B_x = \{ v \in T_x M \ | \ g(Jv, w) = \omega(Jv, w) = 0, \forall w \in A_x \} = B_x.
\end{equation}

Moreover, $J$ is constant along $\mathcal{F}$ (\ie the foliation $\mathcal{F}$ is holomorphic): For any $v = X_a + \xi(a) \theta^\omega \in \mathcal{F}$,
\begin{equation}
\label{eq:thmLCKred2}
\mathcal{L}_v J = \mathcal{L}_{X_a} J + \xi(a) \mathcal{L}_{\theta^\omega} J = 0.
\end{equation}
since $J$ is $G$-invariant and we assumed $\theta^\omega$ is holomorphic.

Note that $\pi_{*|B_x}:B_x \to T_{\pi(x)}N_\xi$ is an isomorphism. Equations \eqref{eq:thmLCKred1} and \eqref{eq:thmLCKred2} then allow us to define a complex structure $J_\xi$ on $N_\xi$:
\[
(J_\xi)_{\pi(x)} (v) = \pi_{*|B_x} (J_x (\pi_{*|B_x})^{-1}(v)).
\]

We can now define $g_\xi$ on $N_\xi$ by $g_\xi (\cdot, \cdot) := \omega_\xi (\cdot, J_\xi \cdot)$. Its pullback is
\begin{equation*}
\begin{split}
(\pi^* g_\xi)_x (v, w) &= (g_{\xi})_x (\pi_* v, \pi_* w) = (\omega_{\xi})_x (\pi_* v, J_\xi \pi_* w) \\
&= (\omega_{\xi})_x (\pi_* v, \pi_* Jw) = (\pi^* \omega_\xi)_x (v, Jw) \\
&= e^{f(x)} \omega_x (v, Jw) = e^f(x) g(v, w), \ \forall v, w \in B_x,
\end{split}
\end{equation*}
thus $\pi^* g_\xi = e^f g_{|\mu^{-1}(\xi)}$ on the horizontal space $B_x$. In particular, $g_\xi$ is a Riemannian metric.

\hfill

We now only have to show that $J_\xi$ is in fact integrable. But this follows from the Newlander-Nirenberg Theorem (see \cite{nn}), as we can relate the Nijenhuis tensors of $J$ and $J_\xi$ \textit{via} the projection:
\[
\pi_* N_J (v, w) = N_{J_\xi}(\pi_* v, \pi_* w), \ \forall v, w \in B_x. 
\] 
\end{proof}
\end{theorem}

\bigskip

To make the reduction work in the LCK setting, we had to add two extra hypotheses: $G_\xi = G$ and that the s-Lee field is holomorphic. 

The first was to ensure that we can find a $J$-invariant subspace (denoted by $B_x$ above). This condition is the same for \K \ reduction done without resorting to the `shifting trick', and is assured, for example, if $G$ is abelian.

The second is unique to the locally conformal setting, and is satisfied for a large subclass of LCK manifolds:

\begin{definition}
	\label{def:vaisman}
	An LCK manifold $(M, J, g, \omega, \theta)$ is called \textit{Vaisman} if
	\[
	\nabla \theta = 0,
	\] 
	where $\nabla$ is the Levi-Civita connection of $g$. 
\end{definition}

\begin{proposition}(see \eg \cite{do})
\label{prop:equivVaisman}
Equivalently, an $LCK$ manifold $(M, J, g, \omega, \theta)$ is Vaisman if $\mathcal{L}_{\theta^\sharp} g = 0$ (the Lee field is Killing).
\end{proposition}

\bigskip

For Vaisman manifolds, the Lee field $\theta^\sharp$ is holomorphic (see \eg \cite{do}). But, as remarked in \eqref{eq:thetasharp&omega}, $\theta^\omega = J \theta^\sharp$, and the set of holomorphic vector fields is closed under $J$ (this can be seen by applying the Nijenhuis tensor).

On the other hand, under certain conditions, the converse is also true. In a recent result, A. Moroianu, S. Moroianu and Ornea \cite{mmo} showed that a compact LCK manifold whose Lee field is holomorphic and satisfies one of two conditions must be Vaisman:

\begin{theorem}\textnormal{(\cite{mmo})}
Let $(M, J, g, \theta)$ be a compact LCK manifold with holomorphic Lee field $\theta^\sharp$. Suppose that one of the following conditions is satisfied:
\begin{enumerate}[(i)]
	\item The norm of the Lee form $\theta$ is constant, or
	\item The metric $g$ is Gauduchon (which means by definition that the Lee form $\theta$ is co-closed with respect to $g$).
\end{enumerate}
Then $(M, J, g)$ is Vaisman.
\end{theorem}

\bigskip

Observe that the first hypothesis is satisfied and the second is unneeded for the regular value $\xi = 0$. The reduction then coincides with the one already studied by Gini, Ornea and Parton \cite{gop}.

We end the section by studying the conditions under which the Vaisman structure on $M$ carries over to the reduced space. It turns out that an additional condition about the group action is necessary:

\begin{corollary}
\label{lem:VaismanRed}
Under the conditions of \ref{thm:LCKred}, if $M$ is Vaisman, the action of $G$ preserves the Riemannian metric (equivalently, preserves the LCK form), and the Lee field $\theta^\sharp$ is tangent to $\mu^{-1}(\xi)$, then $(N_\xi, J_\xi, g_\xi, \theta_\xi)$ is also Vaisman.
\end{corollary}

\begin{proof}
There are a several ways to prove that, in this case, the defining equation of the Vaisman manifold can be pushed down to the quotient. We give one below:

Since $G$ preserves the LCK form, recall from \ref{thm:LCKred} that we have $\pi^* \omega_\xi = \omega_{| \mu^{-1}(\xi)}$, which implies $\pi^* g_\xi = g_{|\mu^{-1}(\xi)}$ on the horizontal space and $\pi^* \theta_\xi = \theta_{| \mu^{-1}(\xi)}$.

According to \ref{prop:equivVaisman}, we need to prove that $\mathcal{L}_{\theta_\xi^\sharp} g_\xi = 0$ (the dual here is with respect to $g_\xi$).

Take $x \in \mu^{-1}(\xi)$. Recall the previous notations: $A_x = (T_x\mu^{-1}(\xi))^\omega$ and $B_x = A_x^\perp \cap T_x\mu^{-1}(\xi)$. Notice that, since $\theta_{|A_x} = 0$ and $\theta^\sharp_x \in T_x \mu^{-1}(\xi)$, we have $\theta^\sharp \in B_x$ \ie the Lee field is horizontal with respect to the submersion. 

Moreover, for a $v \in B_x$,
\begin{equation*}
\begin{split}
(g_\xi)_{\pi(x)} (\pi_* \theta^\sharp_x, \pi_* v) &= (\pi^* g_\xi)_x (\theta^\sharp_x, v) = g_x (\theta^\sharp_x, v) = \theta_x(v) \\ 
&= (\theta_\xi)_{\pi(x)} (\pi_* v),
\end{split}
\end{equation*}
\ie the Lee field is projectable and $\pi_* \theta^\sharp = \theta_\xi^\sharp$.

Take two vector fields $X$ and $Y$ tangent to $\mu^{-1}(\xi)$ which are projectable and horizontal (since $\pi$ is a surjective Riemannian submersion, every vector field on $N_\xi$ is the projection of exactly one such vector field). Then
\[
\mathcal{L}_{\theta_\xi^\sharp} g_\xi (\pi_* X, \pi_* Y) = \mathcal{L}_{\pi_* \theta^\sharp} g_\xi (\pi_* X, \pi_* Y) = \mathcal{L}_{\theta^\sharp} g (X, Y) = 0,
\]
which concludes the proof.
\end{proof}

\bigskip

In subsection \ref{subsec:sasaki}, we will study a class of Vaisman manifolds and Lie group actions for which the conditions above are satisfied. 

\section {Contact reduction}
\label{sectionContRed}

Contact reduction has been studied by many authors, the earliest work being by Guillemin and Sternberg \cite{gs}, referring  to reduction of symplectic cones. Geiges \cite{ge} showed that, under certain conditions, the quotient of the zero level set of the momentum mapping by the group action was naturally endowed with a contact structure. Lerman and Willett \cite{lw} studied the topological structure of contact quotients and later Willett \cite{wi} gave a new contact reduction method that, under certain conditions and factoring not the level sets of the momentum mapping but preimages of rays through regular values, produced a quotient with a contact structure that depended only on the initial contact structure and not on the form itself. Earlier, Albert \cite{al} had given another type of contact reduction that did factor level sets of the momentum mapping, much like the symplectic case, this construction however depending crucially on a choice of contact form in the contact structure (see \cite[Appendix]{wi}). 

Applying our LCS reduction described in Section \ref{sectionLcsRed} to a contact setting is the next logical step and will come naturally, as we shall see below. Comparing it with Albert's reduction (originally given intrinsically), we will find that the resulting quotient equipped with a contact form is the same (again, see \cite[Appendix]{wi}), although obtained by different methods.

We now apply the results of the previous section to the contact context. 

Consider an action of a Lie group $G$ on a contact manifold $(C, \alpha)$ such that $\alpha$ is $G$-invariant.
\begin{definition}
\label{def:contactMM}
The \textit{contact momentum mapping} is defined as
\[
\mu_C : C \to \mathfrak{g}^*, \mu_C (x)(a) = \alpha_x (X_a),
\]
where, as in the previous section, $X_a$ is the fundamental vector field, given by the action, corresponding to $a \in \mathfrak{g}$.
\end{definition}

\begin{claim}
\label{rem:contactMMEchvar}
The contact momentum mapping is equivariant.

\begin{proof}
	This is a standard check using only the definition:
	\begin{align*}
	\mu_C (g \cdot x)(a) &= \alpha_{g \cdot x} (X_a(g \cdot x)) = \alpha_{g \cdot x} (g_* g_*^{-1} X_a(g \cdot x)) \\
	&= (g^* \alpha)_x (g_*^{-1} X_a(g \cdot x)) = \alpha_x (g_*^{-1} X_a(g \cdot x)) \\
	&= \alpha_x \left( \frac{d}{dt}\restrict{t=0} ((g^{-1} \exp(ta) g) \cdot x) \right) \\
	&= \alpha_x \left( \frac{d}{dt}\restrict{t=0} ((g^{-1} \exp(ta) g) \cdot x) \right) \\
	&= \alpha_x ( X_{\Ad(g^{-1})(a)} ) = (g \cdot \mu_c(x))(a).
	\end{align*}
	\end{proof}
\end{claim}

\subsection{Contact reduction and compatibility with LCS reduction}
\label{subsec:contactred}

Let $(C, \alpha)$ be a contact manifold. Then $M = S^1 \times C$ carries an LCS structure, given by
\[
\omega = d_\theta \alpha
\]
where the Lee form $\theta$ is the angular form on $S^1$ (see \ref{prop:lcsCont}); we commit a slight abuse by neglecting the projection mappings. 

Extend the action of $G$ on $C$ to $M = S^1 \times C$ trivially:
\[
g \cdot (t, x) = (t, g \cdot x),\ \ \ t \in S^1, \ x \in C.
\]
This is a twisted symplectic action, since it actually preserves the LCS form. As there is no risk of confusion, we  denote the new fundamental vector fields $X_a$, too. In fact, 
\[
0 = \mathcal{L}^\theta_{X_a} \alpha = d_\theta (\alpha(X_a)) + i_{X_a} \omega,
\]
so we see that the action is twisted Hamiltonian with the momentum mapping
\begin{equation}
\label{eq:momMappingRelation}
\mu : M \to \mathfrak{g}^*, \ \ \ \mu(t, x) = -\mu_C(x).
\end{equation}

Consequently, the level sets are  products: 
\[
\mu^{-1}(\xi) = S^1 \times \mu_C^{-1}(-\xi).
\]
Now assume that $-\xi$ is a regular value for the contact momentum mapping $\mu_C$ (equivalently, that $\xi$ is a regular value for the LCS momentum mapping). If the other hypotheses of \ref{th:LCSred} are met (namely, that the quotient $\faktor{\mu^{-1}(\xi)}{\mathcal{F}}$ is a manifold and the projection is a submersion), this creates an LCS manifold $N_\xi$. Denote by $(\omega_\xi, \theta_\xi)$ its LCS couple. We then have that $\pi^* \omega_\xi = \omega_{|\mu^{-1}(\xi)}$ and $\pi^* \theta_\xi = \theta_{|\mu^{-1}(\xi)}$.

\begin{claim} $N_\xi$ is a product.
\end{claim}

\begin{proof}Indeed, recall the proof of \ref{th:LCSred}: the foliation $\mathcal{F}$ has elements of the form
\[
v = X_a + \xi(a) \theta^\omega.
\]
But $X_a$ is tangent to $C$, and so is $\theta^\omega$, since 
\[
\theta(\theta^\omega) = 0 \text{ if and only if } \theta^\omega \in \Ker \ \theta = TC,
\]
so $\mathcal{F}$ is tangent to $C$. Denote by
\[
\phi_t : S^1 \times C \to S^1 \times C, \ \phi_t (s, x) = (s + t, x);
\]
since $\phi_t^* \omega = 0$, $\mathcal{F}_{| \{t\} \times C}$ is independent of $t$ as a foliation on $C$ - denote it by $\mathcal{F}_C$.

We come to the conclusion that the quotient $N_\xi$ is itself of product form, say
\[
N_\xi = S^1 \times C_\xi,
\]
and 
\[
C_\xi = \faktor{C}{\mathcal{F}_C}.
\]
\end{proof}

\begin{claim}
$C_\xi$  has a contact form, descending from $\alpha$ on $C$. 
\end{claim}

\begin{proof}
	
{\bf Step 1. $\alpha$ descends to $C_\xi$.}

We need to prove that, for a $v \in \mathcal{F}_C$, we have $\alpha(v) = 0$ and $\mathcal{L}_v \alpha = 0$.

In a point $(t, x) \in \mu^{-1}(\xi) = S^1 \times \mu_C^{-1}(-\xi)$, 
\begin{equation}
\begin{split}
\label{eq:contactConcluzie}
\alpha(v) &= \alpha(X_a + \xi(a) \theta^\omega) = \alpha(X_a) + \xi(a) \alpha(\theta^\omega)\\
& = \mu_C(x)(a) + \xi(a) \alpha (\theta^\omega)  = -\xi(a) + \xi(a) \alpha(\theta^\omega).
\end{split}
\end{equation}
Remember that, since $(C, \alpha)$ is a contact manifold, we have a direct sum
\[
T(S^1 \times C) = T S^1 \oplus \Ker\, d\alpha\restrict{C} \oplus \Ker\, \alpha\restrict{C}, 
\]
where the first two terms are $1$-dimensional, the latter being generated by the Reeb field of the contact manifold $C$. Moreover
\begin{equation}
\label{eq:sumaContactNenula}
\alpha(w) = 0 \ \ \text{if and only if}\ \  w = 0,\  \text{ for any } w \in \Ker\, d\alpha\restrict{C}.
\end{equation}

First, since $\Ker\, \theta = TC = \Ker\, d\alpha\restrict{C} \oplus \Ker\, \alpha\restrict{C}$, we obtain $\theta^\omega \in (\Ker\, \theta)^\omega$. But
\begin{equation}\label{eq:ker}
(\Ker\, \theta)^\omega = \Ker\, d\alpha\restrict{C}
\end{equation}
 Indeed, for a $u \in \Ker\, d\alpha\restrict{C}$ and a $w \in \Ker\, \theta$, we have
\[
\omega (u, w) = (d_\theta \alpha) (u, w) = d \alpha (u, w) - (\theta \wedge \alpha) (u, w) = 0,
\]
hence $\Ker\, d\alpha\restrict{C} \subseteq (\Ker\, \theta)^\omega$, and  the equality follows since they are both $1$-dimensional. 

Coming back, we have
\begin{equation}\label{reeb}
\theta^\omega \in (\Ker\, \theta)^\omega = \Ker\, d\alpha\restrict{C} \subset \Ker\, d\alpha.
\end{equation}

Now,
\begin{align*}
\alpha(\theta^\omega) &= \omega(\alpha^\omega, \theta^\omega) = (d_\theta \alpha) (\alpha^\omega, \theta^\omega)  \\
&= d \alpha (\alpha^\omega, \theta^\omega) - (\theta \wedge \alpha) (\alpha^\omega, \theta^\omega)  \\
&= - (\theta \wedge \alpha) (\alpha^\omega, \theta^\omega)  = -\theta(\alpha^\omega)\alpha(\theta^\omega) \\ 
&= (\alpha(\theta^\omega))^2.
\end{align*}
So $\alpha(\theta^\omega)$ satisfies the equation
\[
x^2 - x = 0.
\]
As $\theta^\omega$ has no zeros, \eqref{eq:sumaContactNenula} implies that  $x \neq 0$, and hence
\[
\alpha(\theta^\omega) = 1.
\]
This together with \eqref{reeb} implies that $\theta^\omega$ is the Reeb field.

Coming back to (\ref{eq:contactConcluzie}), we find
\begin{equation}
\label{eq:redCont1}
\alpha(v) = -\xi(a) + \xi(a) \alpha(\theta^\omega) = -\xi(a) + \xi(a) = 0.
\end{equation}

We can now also prove that $\alpha$ is  constant along the flow of a $v \in \mathcal{F}_C$ \ie $\mathcal{L}_v \alpha = 0$.
Indeed, since $\theta(v) = 0$, we have
\begin{equation}
\label{eq:redCont2}
\mathcal{L}_v \alpha = \mathcal{L}^\theta_v \alpha = i_v (d_\theta \alpha) + d_\theta (\alpha(v)) = i_v  \omega\restrict{C} = 0.
\end{equation}

Finally, put (\ref{eq:redCont1}) and (\ref{eq:redCont2}) together: 
\[
\mathcal{L}_v \alpha = \alpha(v) = 0 \text{ on } C \text{ for any } v \in \mathcal{F}_C,
\]
so we can define a $1$-form $\alpha_\xi$ on $C_\xi$ such that $\pi^* \alpha_\xi = \alpha$.

\medskip

{\bf Step 2. $\alpha_\xi$ is a contact form.}

One way to see this is to exploit the relation with the reduced LCS manifold: since $\omega = d_\theta \alpha$, we have
\[
\pi^* (d_{\theta_\xi} \alpha_\xi) = d_{\pi^* {\theta_\xi}} \pi^* \alpha_\xi = d_\theta \alpha = \omega = \pi^* \omega_\xi,
\]
so $d_{\theta_\xi} \alpha_\xi = \omega_\xi$, the very LCS form given by the reduction of \ref{th:LCSred}. This means that $\alpha_\xi$ is a contact form on the manifold $C_\xi$ (see \ref{prop:lcsCont}).
\end{proof}

\begin{remark}There is at least one other way to check equality \eqref{eq:redCont1} which involves less calculations, but doesn't come with the benefit of finding the Reeb field along the way. Notice that $(\Ker \alpha)^\omega = TS^1$. Indeed, for a $u \in TS^1 \subset \Ker\ \alpha$ and a $w \in \Ker\ \alpha$, we have
	\[
	\omega(u, w) = d\alpha(u, w) - \theta \wedge \alpha(u, w) = 0,
	\]
	and the equality follows since both spaces are $1$-dimensional. Then we find that
	\[
	TS^1 \subset T \mu^{-1}(\xi) \ \ \text{implies}\ \  T \mu^{-1}(\xi)^\omega \subset (TS^1)^\omega = \Ker\ \alpha.
	\]
\end{remark}

We gather all we have just proven in the following: 

\begin{theorem}
\label{thm:redcont}
Let $(C, \alpha)$ be a connected contact manifold and $G$ a connected Lie group acting on $C$ and preserving the contact form. Denote by $R$ the Reeb field of $C$. 

Let $\mu_C$ be the momentum mapping and $\xi \in \mathfrak{g}^*$ a regular value. Let $\mathcal{F}_C$ be the foliation
\[
(\mathcal{F}_C)_x = \{ v \in T_x \mu^{-1}(\xi) \ | \ v = (X_a)_x - \xi(a) R_x \text{ for some } a \in \mathfrak{g} \},
\]
where $X_a$ is the fundamental vector field corresponding to $a \in \mathfrak{g}$.

If $C_\xi := \faktor{\mu_C^{-1}(\xi)}{\mathcal{F}_C}$ is a smooth manifold and $\pi : \mu_C^{-1}(\xi) \to C_\xi$ is a submersion, then $C_\xi$ has a natural contact structure such that the contact form $\alpha_\xi$ satisfies
\[
\pi^* \alpha_\xi = \alpha.
\]
Moreover, $(S^1 \times C_\xi, d_\theta \alpha_\xi, \theta)$ is the LCS reduction of $S^1 \times C$ with respect to the regular value $-\xi$.
\end{theorem}

\begin{remark}
\label{rem:albert}
The manifold obtained through \ref{thm:redcont} is the same as the one obtained through Albert's reduction - see \cite{al} and \cite[Appendix]{wi}.
\end{remark}

\begin{remark}
\label{rem:aceeasiRed0}
One can see that this reduction is compatible with the already existing one for regular value $0$ (\ie the one described in \cite{ge}). This is because the LCS reduction for regular value $0$ is just the quotient \textit{via} the group action (see \ref{rem:redzero}).
\end{remark}

\subsection{Sasakian reduction}
\label{subsec:sasaki}

We now show that this reduction method is compatible with the existence of a Sasaki structure on the contact manifold, hence allowing for the definition of a Sasakian reduction for non-zero regular values compatible with the LCK reduction of Section \ref{sectionLCK}. A similar result was proved in \cite{do2} for Willet's \cite{wi} version of contact reduction.

\begin{definition}
\label{def:Sasaki}
A contact manifold $(S, \alpha)$ endowed with a Riemannian metric $g_C$ is called \textit{Sasaki} if $S \times \R$ is \K \ with the natural metric and symplectic form:
\begin{align*}
g &= e^t (g_C + dt^2), \\
\omega &= d(e^t \alpha).
\end{align*}
\end{definition}

\begin{remark}
One can also give an intrinsic definition for a Sasaki manifold - see \eg \cite{bl}.
\end{remark}

\begin{remark}
\label{rmk:SasakiLCK}
Using the above definition of Sasaki manifolds and the property of LCK manifolds to be covered by \K \ manifolds, one can easily prove the following compatibility between the two:

A contact manifold with a Riemannian metric $(C, \alpha, g)$ is Sasaki if and only if $S^1 \times C$ is LCK with the product metric and LCS form:
\begin{align*}
g &= g_C + dt^2, \\
\omega &= d_\theta \alpha.
\end{align*}

From \ref{prop:equivVaisman} we see that, in this case, $S^1 \times C$ is in fact Vaisman.
\end{remark}

\hfill

We use this characterization to give the proof of Sasakian reduction:

\begin{theorem}
\label{thm:redSasaki}
In the conditions of \ref{thm:redcont}, if $C$ is Sasaki, then the reduced space $C_\xi$ carries a natural Sasaki structure.

\begin{proof}
In view of \ref{rmk:SasakiLCK} and \ref{thm:LCKred}, we only need to show that the metric $g_\xi$ on $S^1 \times C_\xi$ is of product type \ie $g_\xi = g_{C_\xi} + dt^2$. But this follows from the way it is constructed: recall that $\pi^* g_\xi = g_{|S^1 \times \mu^{-1}(-\xi)}$ on the horizontal space of the submersion.
\end{proof}
\end{theorem}

\bigskip

Note that we have used the fact that $S^1 \times C$ is Vaisman in order to apply \ref{thm:LCKred}; in turn, since the LCK quotient is also a product of $S^1$ with a Sasakian manifold, it follows that it is also Vaisman. Another way to see this is by applying \ref{lem:VaismanRed}, as Vaisman manifolds of type $S^1 \times S$ with $S$ Sasaki satisfy the additional condition imposed.

\begin{remark}
For $\xi=0$, one recovers the reduction of Sasakian manifolds as introduced in \cite{go} (see also \cite{bs}).
\end{remark}

\section{Examples of LCS reduction}
\label{sectionExamples}

\subsection{Reducing the global conformal structure of $\mathbb{C}^n$}
\label{ssExCn}

Consider the standard \K \ form on $\C^n$, 
\[
\omega_0 = -i \sum_{j=1}^{n} dz_j \wedge d\overline{z}_j = -2\sum_{j=1}^{n} dx_j \wedge dy_j,
\]
and the standard action of $S^1$ on $\C^n$, 
\[
e^{it} \cdot  z = e^{it} \cdot (z_1, ..., z_n) = (e^{it} \cdot z_1, ..., e^{it} \cdot z_n).
\]

We denote by $Y_f$ the Hamiltonian vector field of a smooth function $f$ with regard to the standard metric $\omega_0$.

\begin{proposition}
	This action is Hamiltonian with momentum mapping $\mu: \C^n \to \R, \mu(z) = \| z \|^2.$
	
	\begin{proof}
		Notice that $\omega_0 = d \eta$, where $\eta = -\sum_{i=1}^{n} (x_i  dy_i - y_i dx_i)$ is $S^1$-invariant. Thus, keeping the notations of Section \ref{sectionLcsRed},
		\[
		0 = \mathcal{L}_{X_1} \eta = i_{X_1} \omega_0 + d \eta(X_1), 
		\]
		where $X_1$ is the fundamental vector field corresponding to $1 \in \R$ as the dual of the Lie algebra of $S^1$, $X_1 = \frac{d}{dt} (e^{it} \cdot z) = \sum_{j=1}^n (x_i \frac{\partial}{\partial y_i} - y_i \frac{\partial}{\partial x_i})$, so the action is Hamiltonian and
		\[
		\mu(z) = \rho_1(z) = -\eta(X_1)(z) = \| z \|^2. 
		\]
	\end{proof}
\end{proposition}

\begin{remark}
The symplectic reduction of $(\C^n, \omega_0)$ with respect to this $S^1$-action at regular value $\epsilon > 0$ is $S^{2n-1} / S^1 = \C \mathbb{P}^{n-1}$ with symplectic (in fact \K) form $\epsilon \cdot \omega_{FS}$, where $\omega_{FS}$ is the Fubini-Study metric.

Furthermore, by considering the action of $S^1$ on $\C \times \C^n$,
\[
e^{it} \cdot (w, z) = (e^{-it} w, e^{it}z),
\]
which is still Hamiltonian with momentum mapping $\mu(w, z) = \| z \|^2 - |w|^2$, the reduction at $\epsilon > 0$ is the blow-up of $\C^n$ at the origin (see \cite{ler}).
\end{remark}

These quotient manifolds are easy to understand because the momentum mappings have nice level sets. However, the LCS reduction techniques we developed allow us to greatly expand the range of codimension-$1$ submanifolds of $\C^n$ that can be factored to \textit{produce symplectic manifolds}:

\begin{proposition}
\label{thm:ExSymCn}
Let $f: \C^n \to \R$ be a smooth function and $\xi > 0$ such that 
\begin{equation}
\label{eq:ExSymCn1}
e^{f(z)} \| z \|^2 = \xi \implies d_xf \neq -2\frac{z}{\| z \|^2}, \text{ for all } z \in \C^n.
\end{equation}

If the level set $\mu_f^{-1}(\xi) = \{ z \in \C^n \ | \ e^{f(z)} \| z \|^2 = a \}$ factored along the foliation
\[
\mathcal{F}_z = \langle X_1 + \xi e^{-f(z)}Y_f \rangle
\]
is a smooth manifold $N_\xi$, then $N_\xi$ has a unique symplectic structure $\omega_\xi$ satisfying $\pi^* \omega_\xi = \omega_0$.

\begin{proof}
This is a consequence of \ref{th:LCSred}:

Consider the globally conformally symplectic form $\omega_f = e^f \omega_0$ on $\C^n$. Then the action of $S^1$ is twisted Hamiltonian with momentum mapping $\mu_f = e^f \| z \|^2$. Condition $\eqref{eq:ExSymCn1}$ just means that $\xi$ is a regular value of $\mu_f$.

In this particular case, the characteristic foliation becomes
\begin{equation*}
\begin{split}
T\mu_f^{-1}(\xi) \cap (T\mu_f^{-1}(\xi))^{\omega_f} &= (T\mu_f^{-1}(\xi))^{\omega_f} \\
&= \{ X_a + (\xi \cdot a) \theta^{\omega_f} \ | \ a \in \mathbb{R} \} \\
&= \langle X_1 + \xi e^{-f}Y_f \rangle,
\end{split}
\end{equation*}
since $\theta = df$ and $(df)^{\omega_f} = e^{-f} (df)^{\omega_0} = e^{-f} Y_f$.

The additional hypotheses of the theorem are met: $\theta = df$ is obviously exact and the bilinear forms $\xi \wedge \theta_z(X_\cdot)$ are trivially zero, since they are defined on a $1$-dimensional space. It follows that the LCS form $\omega_\xi$ on $N_\xi$ is in fact symplectic, since it satisfies $\pi^* \omega_\xi = e^{-f} e^f \omega_0 = \omega_0$.
\end{proof}
\end{proposition}

\subsection{Reducing Hopf manifolds}

We now look at the reduction of LCS Hopf manifolds \ie $S^1 \times S^{2n-1}$. 

The standard action of $S^1$ on $\C^n$ described in Subsection \ref{ssExCn} can be restricted to the standard $S^1$-action on the contact manifold $S^{2n-1} \subset \C^n$. It then follows (see \eg \cite[Example 2.7]{wi}) that the contact momentum mapping is just the restriction of $\mu$, so the contact sphere \textit{has constant momentum mapping} $\mu_C = 1$.

Coming to the Hopf manifold $H = (S^1 \times S^{2n-1}, \omega = d_\theta \alpha)$ with the $S^1$ action on the second factor, we obtain from \eqref{eq:momMappingRelation} that its momentum mapping is constant $\mu_0 \equiv -1$. This means that neither reduction results in  \ref{th:LCSred} or \ref{thm:redcont} are directly applicable, because the only interesting value $\xi = -1$ is not regular.

However, recall that \ref{th:LCSred} works even if the LCS form is not invariant under the group action, so we may replace $\omega$ by another form in its conformal class $\omega_f = e^f \omega$ with momentum mapping $\mu_f = e^f \cdot (-1) = -e^f$ (see \ref{rem:changemu}). This in turn gives a great flexibility in terms of the level set:

\begin{remark}
\label{rem:levset}
For any manifold $M$ and any closed hypersurface $N \subset M$, there exists an $f \in \mathcal{C}^\infty(M)$ having $0$ a regular value and $N = f^{-1}(0)$.

\begin{proof}
Indeed, we can take $f = (d(\cdot, N))^2$, where $d$ is the distance induced by the LCK metric on $H$ (or any other Riemannian metric).
\end{proof}
\end{remark}

\bigskip

Take $f \in \mathcal{C}^\infty(S^1 \times S^{2n-1})$ with $0$ a regular value (equivalently, $-1$ is a regular value for $\mu_f = -e^f$). Then the foliation that $\mu_f^{-1}(-1)$ must be quotiened to, according to \eqref{eq:ortho}, is
\[
\mathcal{F}_x = (T_x \mu^{-1}(-1))^{\omega_f} = \langle X_1 + \xi(1) (\theta + df)^{\omega_f}\rangle_\R = \langle X_1 -  (\theta + df)^\omega\rangle_\R,
\]
for all $x \in f^{-1}(0)$ (we took into account that $f = 0$ on $\mu_f^{-1}(-1)$ by definition). This can be further simplified: recall from Subsection \ref{subsec:contactred} that $\theta^\omega$ is the Reeb field of $S^{2n-1}$, $\theta^\omega = X_1 = R$, so
\begin{equation}
\label{eq:exHopfFoliatie}
\mathcal{F}_x = \langle Y_f\rangle,
\end{equation}
where, again, $Y_f$ is the Hamiltonian vector field of $f$ with regard to the canonical symplectic form $\omega$ on $H$.

We combine this fact with the LCS reduction \ref{th:LCSred} to obtain

\begin{proposition}
\label{thm:redHopf}
Let $f:( S^1 \times S^{2n-1}, \omega, \theta) \to \R$ be a smooth function with $0$ a regular value. Assume $\theta$ is exact along the foliation $\mathcal{F} = \langle Y_f\rangle$ on $f^{-1}(0)$.

If  $\faktor{f^{-1}(0)}{\mathcal{F}}$ is a smooth manifold $N$, then $N$ has a unique LCS structure $\omega_N$ satisfying $\pi^* \omega_N = \omega$.
\end{proposition}

\begin{remark}
	\label{rem:redHopf0}
	A special case of $f: S^1 \times S^{2n-1} \to \mathbb{R}$ satisfying the above criteria is one for which $\theta(Y_f) = 0$. This means
	\begin{equation}
	\label{eq:rem:redHopf0_1}
	(\Ker \ df)^\omega \subset \Ker \ \theta \iff (\Ker \ \theta)^\omega \subset \Ker \ df.
	\end{equation}
	But $(\Ker \ \theta)^\omega = \langle X_1 \rangle$, so, looking at the Hopf fibration $H$:
	
	\begin{center}
		\begin{tikzpicture}[start chain] {
			\node[on chain] {$0$};
			\node[on chain] {$S^1$};
			\node[on chain] {$S^{2n-1}$};
			\node[on chain, join={node[above] {$p$}}] {$\mathbb{CP}^{n-1}$};
			\node[on chain] {$0$}; }
		\end{tikzpicture}
	\end{center} 
	condition (\ref{eq:rem:redHopf0_1}) means that $f = g \circ (id, p)$ for some $g:S^1 \times \mathbb{CP}^{n-1} \to \mathbb{R}$ having $0$ a regular value. If $G = g^{-1}(0) \subset S^1 \times \mathbb{CP}^{n-1}$, then $f^{-1}(0)$ is the total space of the fibration $(S^1 \times H)_{|G}$.
\end{remark}

%
%\begin{corollary}
%TBD

%\begin{proof}
%Take $M \subset \mathbb{CP}^n$ a closed hypersurface. By \ref{rem:levset}, there is a function $\tilde{g}: \mathbb{CP}^n \to \mathbb{R}$ such that $M = \tilde{g}^{-1}(0)$. Let $\tilde{f} = \tilde{g} \circ p$ and $g = S^1 \times \mathbb{CP}^n \to \mathbb{R}$,
%\[
%g(e^{it}, x) = \tilde{g}(x).
%\]

%Take $f = g \circ p$ as in the previous remark. Then 
%\[
%f^{-1}(0) = S^1 \times \tilde{f}^{-1}(0).
%\]
%Furthermore, since $f$ comes from a function on $\mathbb{CP}^n$, the direction of the Hamiltonian of $f$ can be determined:
%\[
%TS^1 \cup \langle X_1 \rangle \subset \Ker df \iff \langle Y_f \rangle = (\Ker df)^\omega \subset 
%\]

%\end{proof}
	
%\end{corollary}

\begin{example}
Consider $g:S^1 \times \mathbb{C}\mathbb{P}^2 \to \mathbb{R}$,
\[
g(e^{it}, [z_1:z_2:z_3]) = \frac{\Re (z_1 \overline{z}_2)}{|z_1|^2 + |z_2|^2 + |z_3|^2},
\]
for which $0$ is a regular value. Let $f = g \circ p:S^1 \times S^3 \to \mathbb{R}$ as above. Then 
\[
f^{-1}(0) = S^1 \times \{ (z_1, z_2, z_3) \in S^5 \ | \ \Re (z_1 \overline{z}_2) = \langle z_1, z_2 \rangle = 0 \}
\]
and $\faktor{f^{-1}(0)}{Y_f}$ has an LCS structure, according to \ref{thm:redHopf}.

%Denote by $R, A, B \in \mathcal{X}(S^3)$ the canonical linearly independent vector fields on $S^3$, $R = X_1$ being the Reeb vector field. According to the previous remark, we have $Y_f \in \Ker \ \theta = TS^3$. Notice that $TS^1 \subset \Ker df \iff Y_f \in (TS^1)^\omega = TS^1 \oplus \langle A, B \rangle$. We obtain
%\[
%Y_f \in \langle A, B \rangle. 
%\]

\end{example}

\subsection{Reducing the cotangent bundle}

Another example of manifolds where LCS reduction is applicable are the cotangent bundles:

\begin{remark}(\cite{hr2})
\label{rem:cotangentLCS}
In addition to the canonic symplectic structure, the cotangent bundle of a given manifold $Q$ has various LCS structures: consider the tautological $1$-form $\eta$ on $T^*Q$, defined in any point $\alpha_x \in T^*_xQ$ by
\[
\eta_{\alpha_x} (v) = \alpha_x (\pi_* v).
\]
Then, for any $\theta$ a closed $1$-form on $Q$, $\omega_\theta = d_{\pi^* \theta} \eta$ is an LCS form on $T^*Q$ (which is not globally conformally symplectic unless $\theta$ is exact).
\end{remark}

We start by proving a few facts about extending Lie group actions on $Q$ to actions on its cotangent bundle:

\begin{proposition}
\label{prop:cotangentAct}
Let $Q$ be $n$-dimensional smooth manifold with a (left) Lie group action of $G$. Suppose $Q$ and $G$ are connected. Denote by $X_a \in \mathcal{X}(Q)$ the fundamental vector field corresponding to an $a \in \mathfrak{g}$.

Pick $\theta \in \Omega^1(Q)$ closed and suppose $\theta(X_a) = 0$, for all $a \in \mathfrak{g}$.

Then:
\begin{enumerate}[(i)]
	\item There exists a natural action of $G$ on $T^*Q$. Denote by $\tilde{X}_a \in \mathcal{X}(T^*Q)$ its fundamental vector fields.
	\item $\tilde{X}_a$ and $X_a$ are $\pi$-related: $\pi_*(\tilde{X}_a) = X_a$.
	\item The action is twisted Hamiltonian with respect to the LCS structure $(T^*Q,  \omega_\theta, \theta)$. Moreover, it preserves $\omega_\theta$.
	\item The momentum mapping of this actions is $\mu(\alpha_x)(a) = -\alpha_x (X_a)$.
	
\end{enumerate}

\begin{proof}
	
\begin{enumerate}[(i)]
	\item We define the (left) action of $G$ on $T^*Q$ by push-forward: 
	\[
	g \cdot \alpha_x = g_* (\alpha_x) \in T_{g \cdot x}^* Q, \ \forall \alpha_x \in T_x^*Q.
	\]
	Note that $\eta$ is $G$-invariant:
	\begin{equation*}
	\begin{split}
	(g^* \eta)_{\alpha_x}(v) &= \eta_{g \cdot \alpha_x} (g_* v) = (g \cdot \alpha_x)(\pi_* g_* v) = (g_* \alpha_x) ((\pi g)_* v) \\
	&= (g_* \alpha_x) ((g \pi)_* v) = \alpha_x (g^{-1}_* g_* \pi_* v) = \alpha_x (\pi_* v) \\
	&= \eta_{\alpha_x}(v), \ \forall v \in T_{\alpha_x} T^*Q.
	\end{split}
	\end{equation*}
	\item Indeed,
	\begin{equation*}
	\begin{split}
	\pi_*((\tilde{X}_a)_{\alpha_x}) &= \pi_*(\frac{d}{dt}_{|t = 0} (e^{ta} \cdot \alpha_x)) = \frac{d}{dt}_{|t = 0} (\pi (e^{ta} \cdot \alpha_x)) \\
	&= \frac{d}{dt}_{|t = 0} ( e^{ta} \cdot (\pi \alpha_x)) = \frac{d}{dt}_{|t = 0} ( e^{ta} \cdot x) = (X_a)_x.
	\end{split}
	\end{equation*}
	\item Denote $\tilde{\theta} = \pi^* \theta$. Then
	\[
	\mathcal{L}_{\tilde{X}_a} \tilde{\theta} = d (\tilde{\theta}(\tilde{X}_a)) +  i_{{\tilde{X}_a}}d\tilde{\theta} = d (\theta(X_a)) = 0.
	\] 
	Together with the fact that $\eta$ is also $G$-invariant, we obtain that $\omega_\theta = d_{\tilde{\theta}} \eta$ is $G$-invariant. Furthermore,
	\[
	0 = \mathcal{L}^{\tilde{\theta}}_{\tilde{X}_a} \eta = d_{\tilde{\theta}} (\eta(\tilde{X}_a)) + i_{\tilde{X}_a} d_{\tilde{\theta}} \eta = d_{\tilde{\theta}} (\eta(\tilde{X}_a)) + i_{\tilde{X}_a} \omega_\theta,
	\]
	so the action is twisted Hamiltonian with $\rho_a(\alpha_x) = -\eta_{\alpha_x} (\tilde{X}_a)$.
	\item It then follows by definition that $\mu(\alpha_x)(a) = -\alpha_x (X_a)$.
\end{enumerate}
\end{proof}
\end{proposition}

This leads directly to the following consequence of the LCS reduction theorem \ref{th:LCSred}:

\begin{proposition}
\label{thm:redCotan}
Let $Q$ be a connected manifold and $G$ a connected Lie group acting on it. Let $\theta \in \Omega^1(Q)$ be closed and suppose $\theta(X_a) = 0$, for all $a \in \mathfrak{g}$.

Let $\mu$ be the momentum mapping of the corresponding twisted Hamiltonian action of $G$ on $T^*Q$ and $\xi \in \mathfrak{g}^*$ a regular value. Let $\mathcal{F} = T\mu^{-1}(\xi) \cap (T\mu^{-1}(\xi))^\omega$.

If $N_\xi := \faktor{\mu^{-1}(\xi)}{\mathcal{F}}$ is a smooth manifold and $\pi : \mu^{-1}(\xi) \to N_\xi$ is a submersion, then $N_\xi$ has an LCS structure such that the LCS form $\omega_\xi$ satisfies
\[
\pi^* \omega_\xi = {\omega_\theta}_{|\mu^{-1}(\xi)}.
\]
\end{proposition}

\begin{example}
Let $Q = S^1 \times S^3$ with the induced Riemannian metric $g_0$ and the standard action of $S^1$ on the second factor. Since $S^3$ is parallelizable, $T^*(S^1 \times S^3) \simeq S^1 \times S^3 \times \mathbb{R}^4$; denote by $R, A, B \in \mathcal{X}(S^3)$ the canonical linearly independent vector fields on $S^3$ arising from the identification of $\mathbb{R}^4$ with the quaternion field $\mathbb{H}$ (\ie $R_q = i \cdot q, A_q = j \cdot q$ and $B_q = k \cdot q$ for any $q \in \mathbb{H}$), $R = X_1$ being the Reeb vector field, and by $V$ the unitary vector field on $S^1$.

\begin{remark}
\label{rem:calculeQuaternion}
\begin{enumerate}
	\item For any quaternion $q \in S^3$, the standard $S^1$-action is simply the multiplication in $\mathbb{H}$.
	\item The vector fields $R$ and $V$ are $S^1$-invariant. 
	\item For any $t$ and any $q \in S^3$, we have 
	\begin{equation*}
	\begin{split}
	(e^{it})_* A_q &= \cos 2t \cdot A_{e^{it} \cdot q} + \sin 2t \cdot B_{e^{it} \cdot q} \\
	(e^{it})_* B_q &= - \sin 2t \cdot A_{e^{it} \cdot q} + \cos 2t \cdot B_{e^{it} \cdot q}. 
	\end{split}
	\end{equation*}
	
\end{enumerate}

\begin{proof}
\begin{enumerate}
	\item \label{magnus} Indeed, we have the identification 
	\[
	\mathbb{C}^2 \ni (z_1, z_2) \leftrightarrow z_1 + j \overline{z_2} \in \mathbb{H}.
	\]
	Then 
	\begin{equation*}
	\begin{split}
	e^{it} \cdot q &= e^{it} \cdot (z_1 + j \overline{z_2}) = e^{it} \cdot (z_1, z_2) = (e^{it} z_1, e^{it} z_2) \\ &= e^{it} z_1 + j e^{-it} \overline{z_2} = e^{it} (z_1 + j \overline{z_2}).
	\end{split}
	\end{equation*}
	\item This is obvious, as $S^1$ acts only on $S^3$ and $R$ is a fundamental vector field of the action.
	\item Using (\ref{magnus}), we calculate the push-forward of $A_q$ by the action of $e^{it}$ as a free vector in $\mathbb{R}^4$, since the multiplication by $e^{it}$ is linear:
	\begin{equation*}
	\begin{split}
	e^{it} \cdot A_q &= e^{it} \cdot jq = e^{it} je^{-it} (e^{it}q) = e^{2it} j (e^{it}q) \\
	&= \cos2t \cdot j (e^{it}q) + i \sin 2t \cdot j (e^{it}q) \\
	&= \cos 2t \cdot A_{e^{it} \cdot q} + \sin 2t \cdot B_{e^{it} \cdot q}.
	\end{split}
	\end{equation*}
	
	Similarly,
	\begin{equation*}
	\begin{split}
	e^{it} \cdot B_q &= e^{2it} k (e^{it}q) = \cos2t \cdot k (e^{it}q) + i \sin 2t \cdot k (e^{it}q) \\
	&= \cos 2t \cdot B_{e^{it} \cdot q} - \sin 2t \cdot A_{e^{it} \cdot q}.
	\end{split}
	\end{equation*}
	
\end{enumerate}
\end{proof}

\end{remark}

Consider the fibration

\begin{center}
\begin{tikzpicture}[start chain] {
	\node[on chain] {$0$};
	\node[on chain] {$S^1 \times S^1$};
	\node[on chain] {$S^1 \times S^3$};
	\node[on chain, join={node[above] {$p$}}] {$S^1 \times S^2$};
	\node[on chain] {$0$}; }
\end{tikzpicture}
\end{center}

Let $\theta \in \Omega^1(S^1 \times S^3)$ be closed such that $\theta(R) = 0$. This means that $\theta = p^* \eta$, for some closed form $\eta$ on $S^1 \times S^2$.

According to \ref{rem:calculeQuaternion} and since the metric $g_0$ is $S^1$-invariant, the corresponding action on $T^*(S^1 \times S^3) \simeq S^1 \times S^3 \times \langle V^\sharp, R^\sharp, A^\sharp, B^\sharp \rangle$ is 

\begin{equation}
\label{eq:actpecotang}
(z, x, (v, r, a, b)) \xrightarrow{e^{it} \cdot } (z, e^{it} \cdot x, (v, r, e^{-2it}\cdot(a, b))).
\end{equation}

\ie it induces a rotation of the opposite angle on the $\langle A^\sharp, B^\sharp \rangle$ plane.

According to \ref{prop:cotangentAct}, the momentum mapping of the $S^1$ action on $T^*(S^1 \times S^3)$ is 
\begin{equation*}
\begin{split}
\mu &: T^*(S^1 \times S^3) \to \mathbb{R}, \\
\mu(v V_x^\sharp + r R_x^\sharp + a A_x^\sharp + b B_x^\sharp) &= - (v V_x^\sharp + R_x^\sharp + a A_x^\sharp + b B_x^\sharp)(R_x) = -r
\end{split}
\end{equation*}
\ie the projection (with reversed sign) on the first factor of $\mathbb{R}^4$. In particular, every value is regular.

Consider the quotient of the $0$-level set by the group action described in \eqref{eq:actpecotang}. We have
\[
\mu^{-1}(0) = S^1 \times S^3 \times \langle V^\sharp \rangle \times \langle A^\sharp, B^\sharp \rangle, 
\]
so
\[ \faktor{\mu^{-1}(0)}{S^1} = S^1 \times \mathbb{R} \times \faktor{(S^3 \times \langle A^\sharp, B^\sharp \rangle)}{S^1} 
\]
as the action is trivial on $S^1$ and its cotangent space. 

We claim $\faktor{(S^3 \times \langle A^\sharp, B^\sharp \rangle)}{S^1} \simeq T^*(\mathbb{CP}^1)$. For this, it is easier to look at the general form of the action on the cotangent bundle (as defined in \ref{prop:cotangentAct}) rather than the concrete formula in \eqref{eq:actpecotang}. Indeed, denote by $\pi: S^3 \to \mathbb{CP}^1$ the Hopf projection and consider the mapping
\begin{equation*}
\begin{split}
\varphi: \faktor{(S^3 \times \langle A^\sharp, B^\sharp \rangle)}{S^1} &\to T^*(\mathbb{CP}^1), \\ \varphi([q, \alpha_q])(w_{[q]}) &= (\alpha_q, v_q), \text{ where } v_q \in \langle A^\sharp, B^\sharp \rangle \text{ with } \pi_* (v_q) = w_{[q]}
\end{split}
\end{equation*}
(here we denoted by $( \ , \ )$ the duality mapping between $T^*_q (\mathbb{CP}^1)$ and $T_q (\mathbb{CP}^1)$).

One can see that $\varphi$ is correctly defined, as
\[
\varphi([e^{it} \cdot q, (e^{-it})^* \alpha_q])(w_{[q]}) = ((e^{-it})^* \alpha_q, (e^{it})_* (v_q)) = (\alpha_q, v_q),
\]
and it is a diffeomorphism.	

It follows from \ref{thm:redCotan} that, for any $\theta$ as above, the reduction at regular value $0$ is $T^* (S^1 \times \mathbb{CP}^1)$ (see \ref{rem:redzero}) with LCS form $\omega_\eta$ and Lee form $\eta$, as in \ref{rem:cotangentLCS}.

Now take $\theta \in \Omega^1(S^1 \times S^3)$ to be the angular form coming from $S^1$. This satisfies $\theta(R) = 0$. One can check (in local coordinates) that $\theta^{\omega_\theta} = V \in TS^1$ the unitary vector field tangent to $S^1$. We can then identify the reduced space up to diffeomorphism using \ref{prop:actiuneGrup}:
\begin{equation*}
\begin{split}
N_r &= \faktor{(S^1 \times \langle V^\sharp \rangle \times S^3 \times \{-r\} \times \langle A^\sharp, B^\sharp \rangle)}{\langle \tilde{X}_1 + r V \rangle} \\
&\simeq S^1 \times \faktor{(\langle V^\sharp \rangle \times S^3 \times \{-r\} \times \langle A^\sharp, B^\sharp \rangle)}{\mathbb{R}},
\end{split}
\end{equation*}
where the action of $\mathbb{R}$ on $\langle V^\sharp \rangle \times (S^3 \times \{-r\} \times \langle A^\sharp, B^\sharp \rangle)$ is defined by 
\[
s \cdot (t, x) = (t + rs, e^{is} \cdot x),
\]
for any $x \in S^3 \times \langle A^\sharp, B^\sharp \rangle$. But then $N_r$ can be understood from a more general remark:

\begin{remark}
\label{rem:RxMfactor}
Let $M$ be a manifold with an $S^1$ action. Define an action of $\mathbb{R}$ on $\mathbb{R} \times M$ by
\[
s \cdot (t, x) = (t + rs, e^{is} \cdot x), \ \forall x \in M,
\]
for some $r \in \mathbb{R}$. Then
\[
\faktor{(\mathbb{R} \times M)}{\mathbb{R}} \simeq M.
\]

\begin{proof}
First, notice that via a normalization, we can assume $r = 1$. Define the mapping
\[
\faktor{(\mathbb{R} \times M)}{\mathbb{R}} \ni \widehat{(t, x)} \to e^{-it} \cdot x \in M.
\]
This can be seen to be a diffeomorphism.
\end{proof}
\end{remark}

We then have, according to \ref{thm:redCotan}, that, for any $r \in \mathbb{R}$, the manifold 
\[
N_r = S^1 \times (S^3 \times \{-r\} \times \langle A^\sharp, B^\sharp \rangle)
\]
has an LCS structure.

\end{example}

\bigskip

\textbf{Acknowledgments.} I would like to thank my advisor, Professor Liviu Ornea, for his guidance and suggestions throughout.


\begin{thebibliography}{100}
	
	\bibitem[A]{al} C. Albert, {\em La th\'eor\`eme de r\'eduction de Marsden-Weinstein en g\'eom\'etrie cosymplectique et de contact}, Journal of Geometry and Physics {\bf 6}, no.4 (1989), 627-649.
	
	\bibitem[AM]{am} R. Abraham, J. E. Marsden, {\em Foundations of Mechanics, Second Edition}, Addison-Wesley Publishing Company, Inc. (1978).
	
	\bibitem[Ba]{baz} G. Bazzoni, {\em Locally conformally symplectic and Kähler geometry}, arXiv:1711.02440.
	
	\bibitem[BGP]{bgp} F. Belgun, O. Goertsches, D. Petrecca, {\em Locally conformally symplectic convexity}, arXiv:1805.00218.
	
	\bibitem[BS]{bs} I. Biswas, G. Schumacher, {\em Symplectic reduction of Sasakian manifolds}, arXiv:1804.03685.
	
	\bibitem[Bl]{bl} D.E. Blair, Riemannian geometry of contact and symplectic manifolds. Second edition. Progress in Mathematics, 203. Birkhäuser Boston, Inc., Boston, MA, 2010. 
	
	\bibitem[Br]{br} R. Bryant, {\em An introduction to Lie groups and symplectic geometry}, Geometry and quantum field theory, IAS/Park City Math. Series 1, American Mathematical society (1995), 5-181.
	
	\bibitem[CM]{mur} B. Chantraine, E. Murphy, {\em Conformal symplectic geometry of cotangent bundles}, arXiv:1606.00861.
	
	\bibitem[D]{dom} W. Domitrz, {\em Reductions of locally conformal symplectic structures and de Rham cohomology tangent to a foliation}, Geometry and topology of caustics (2008), 45-53.
	
	\bibitem[DO]{do} S. Dragomir, L. Ornea, Locally conformally K\"ahler manifolds, Progress in Math. {\bf 55}, Birkh\"auser, 1998.
	
	\bibitem[DrO]{do2} O. Drăgulete, L. Ornea, {\em Non-zero contact and Sasakian reduction}, Differential Geometry and its Applications 24 (2006), 260-270.
	
	\bibitem[EM]{el} Y. Eliashberg, E. Murphy, {\em Making cobordisms symplectic}, arXiv:1504.06312.
	
	\bibitem[G1]{ge} H. Geiges, {\em Constructions of contact manifolds}, Math. Proc. Cambridge Philos. Soc. {\bf 121} (1997), no. 3, 455-464.
	
	\bibitem[G2]{ge2} H. Geiges, {\em An Introduction to Contact Topology}, Cambridge University Press (2008).

	\bibitem[GOP]{gop} R. Gini, L. Ornea, M. Parton, {\em Locally conformal \K \ reduction}, Journal für die reine und angewandte Mathematik (2005), no. 581, 1-21.
	
	\bibitem[GO]{go} G. Grantcharov, L. Ornea, {\em Reduction of Sasakian manifolds}, J. Math. Phys. {\bf 48} (2001), 3809--3816.
	
	\bibitem[GS]{gs} V. Guillemin, S. Sternberg, {\em Homogeneous quaantization and multiplicities of group representations}, Journal of Functional Analysis {\bf 47} (1982), no. 3, 344-380.
	
	\bibitem[HR1]{hr} S. Haller, T. Rybicki, {\em Reduction for locally conformal symplectic manifolds}, Journal of Geometry and Physics {\bf 37} (2001), 262-271.
	
	\bibitem[HR2]{hr2} S. Haller, T. Rybicki, {\em On the group of diffeomorphisms preserving a locally conformal symplectic structure}, Annals of Global Analysis and Geometry {\bf 17}, Issue 5 (1999), 475-502.
	
	\bibitem[I]{nico} N. Istrati, {\em A characterisation of toric LCK manifolds}, arXiv:1612.03832 (2017).
	
	\bibitem[Lee]{lee} H.C. Lee, {\em A kind of even-dimensional differential geometry and its application to exterior calculus}, Amer. J. of Math. {\bf{65}} (1943), 433-438. 
	
	\bibitem[Lef]{lef} J. Lefebvre, {\em Propri\'et\'es du groupe des transformations conformes et du groupe des automorphismes d'une vari\'et\'e  localement conform\'ement symplectique}, C. R. Acad. Sci. Paris S\'er. A--B {\bf 268} (1969) A717--A719.
	
	\bibitem[Ler]{ler} E. Lerman, {\em Symplectic cuts}, Mathematical Research Letters {\bf 2} (1995) 247-258.
	
	\bibitem[LW]{lw} E. Lerman, C. Willett, {\em The topological structure of contact and symplectic quotients}, International Mathematics Research Notices {\bf 2001} (2001), Issue 1, 33-52.
	
	\bibitem[Li]{lib} P. Libermann, {\em Sur le probl\`eme d'\'equivalence de certaines structures infinit\'esimales r\'eguli\`eres}, Ann. Mat. Pura. Appl., {\bf 36} (1954), 27--120.
	
	
	\bibitem[Lo]{lo} F. Loose, {\em Reduction in contact geometry}, Journal of Lie Theory {\bf 11} (2001), no. 1, 9-22.
	
	\bibitem[M]{mil} J. W. Milnor, Remarks on infinite-dimensional Lie groups, Relativity, groups and topology {\bf II}, Les Houches, 1983.
	
	\bibitem[MMO]{mmo} A. Moroianu, S. Moroianu, L. Ornea {\em Locally conformally \K \ manifolds with holomorphic Lee field}, arXiv:1712.0582 (2017).
	
	\bibitem[NN]{nn} A. Newlander, L. Nirenberg, {\em Complex Analytic Coordinates in Almost Complex Manifolds}, Annals of Mathematics {\bf 65} (1957), no. 3, 391-404.
	
	\bibitem[V]{va} I. Vaisman, {\em Locally conformal symplectic manifolds}, Int. J. Math. Math. Sci. {\bf 8} (3) (1985), 521-536. 
	
	\bibitem[W]{wi} C. Willett, {\em Contact reduction}, Trans. Amer. Math. Soc. {\bf 354} (2002), 4245-4260.
	
	\bibitem[ZZ]{zz} M. Zambon, C. Zhu, {\em Contact reduction and groupoid actions}, Trans. Amer. Math Soc. {\bf 358} (2005), No. 3, 1365-1401.
	
	
\end{thebibliography}
\end{document}